\documentclass[10pt,letterpaper,twoside]{article}

\usepackage[letterpaper,left=1in,right=1in,top=1.5in,bottom=1.5in]{geometry}

\title{Report on $\bE_\infty$-descendability}
\author{Benjamin Antieau and Germ\'{a}n Stefanich}
\date{\today}

\usepackage[pdfstartview=FitH,
            pdfauthor={},
            pdftitle={},
            colorlinks,
            linkcolor=reference,
            citecolor=citation,
            urlcolor=e-mail,
            backref]{hyperref}

\usepackage{amsmath}
\usepackage{amscd}
\usepackage{amsbsy}
\usepackage{amssymb}
\usepackage{verbatim}
\usepackage{eufrak}
\usepackage{eucal}
\usepackage{microtype}
\usepackage{hyperref}
\usepackage{mathrsfs}
\usepackage{amsthm}
\usepackage{stmaryrd}
\usepackage{xy}
\input xy
\xyoption{all}
\usepackage{tikz}
\usetikzlibrary{matrix,arrows}
\usepackage{tikz-cd}
\usepackage{dsfont}
\usepackage[T1]{fontenc}
\usepackage{enumitem}
\setlist{noitemsep}

\usepackage{fancyhdr}
\newcommand{\stackspace}{2.5}
\newcommand{\stack}[2][1cm]{\;\tikz[baseline, yshift=.65ex]%
    {\foreach \k [evaluate=\k as \r using (.5*#2+.5-\k)*\stackspace] in {1,...,#2}{%
    \ifodd\k{\draw[->](0,\r pt)--(#1,\r pt);}%
    \else{\draw[<-](0,\r pt)--(#1,\r pt);}\fi
    }}\;}

\usepackage[bbgreekl]{mathbbol}
\usepackage{amsfonts}
\DeclareSymbolFontAlphabet{\mathbb}{AMSb} 
\DeclareSymbolFontAlphabet{\mathbbl}{bbold}

\usepackage{yhmath}

\usepackage{color}
\definecolor{todo}{rgb}{1,0,0}
\definecolor{conditional}{rgb}{0,1,0}
\definecolor{e-mail}{rgb}{0,.40,.80}
\definecolor{reference}{rgb}{.20,.60,.22}
\definecolor{mrnumber}{rgb}{.80,.40,0}
\definecolor{citation}{rgb}{0,.40,.80}

\let\oldmarginpar\marginpar
\renewcommand\marginpar[1]{\-\oldmarginpar[\raggedleft\footnotesize #1]%
{\raggedright\footnotesize #1}}

\newcommand{\Cscr}{\mathcal{C}}
\newcommand{\Dscr}{\mathcal{D}}

\newcommand{\Fscr}{\mathcal{F}}

\newcommand{\Iscr}{\mathcal{I}}

\newcommand{\Oscr}{\mathcal{O}}

\newcommand{\Sscr}{\mathcal{S}}

\newcommand{\Xscr}{\mathcal{X}}

\newcommand{\B}{\mathrm{B}}
\newcommand{\C}{\mathrm{C}}
\renewcommand{\d}{\mathrm{d}}
\newcommand{\D}{\mathrm{D}}

\newcommand{\F}{\mathrm{F}}

\renewcommand{\H}{\mathrm{H}}
\newcommand{\h}{\mathrm{h}}

\newcommand{\R}{\mathrm{R}}

\newcommand{\bA}{\mathbf{A}}

\newcommand{\bE}{\mathbf{E}}
\newcommand{\bF}{\mathbf{F}}

\newcommand{\bN}{\mathbf{N}}

\newcommand{\bQ}{\mathbf{Q}}

\newcommand{\bS}{\mathbf{S}}
\newcommand{\bT}{\mathbf{T}}

\newcommand{\bZ}{\mathbf{Z}}

\newcommand{\afrak}{\mathfrak{a}}

\newcommand{\pfrak}{\mathfrak{p}}
\newcommand{\qfrak}{\mathfrak{q}}

\newcommand{\Fin}{\Fscr\mathrm{in}}

\newcommand{\op}{\mathrm{op}}
\newcommand{\cofib}{\mathrm{cofib}}
\newcommand{\fib}{\mathrm{fib}}

\newcommand{\cn}{\mathrm{cn}}

\newcommand{\Mod}{\mathrm{Mod}}

\newcommand{\QCoh}{\mathrm{QCoh}}
\newcommand{\qc}{\mathrm{qc}}

\newcommand{\CAlg}{\mathrm{CAlg}}
\newcommand{\DAlg}{\mathrm{DAlg}}

\newcommand{\Shv}{\mathrm{Shv}}

\newcommand{\LSym}{\mathrm{LSym}}

\renewcommand{\max}{\mathrm{max}}
\renewcommand{\min}{\mathrm{min}}
\newcommand{\cont}{\mathrm{cont}}

\newcommand{\Cond}{\mathrm{Cond}}
\newcommand{\hyp}{\mathrm{hyp}}
\newcommand{\fin}{\mathrm{fin}}
\newcommand{\Cont}{\mathrm{Cont}}

\newcommand{\ann}{\mathrm{ann}}
\newcommand{\Pro}{\mathrm{Pro}}
\newcommand{\perf}{\mathrm{perf}}

\renewcommand{\part}{\mathrm{part}}

\newcommand{\can}{\mathrm{can}}

\newcommand{\heart}{\heartsuit}

\newcommand{\id}{\mathrm{id}}

\renewcommand{\geq}{\geqslant}
\renewcommand{\leq}{\leqslant}

\DeclareMathOperator{\Ext}{Ext}

\newcommand{\dR}{\mathrm{dR}}

\newcommand{\crys}{\mathrm{crys}}

\newcommand{\KU}{\mathrm{KU}}

\newcommand{\KO}{\mathrm{KO}}

\newcommand{\Map}{\mathrm{Map}}

\newcommand{\Hom}{\mathrm{Hom}}
\newcommand{\Fun}{\mathrm{Fun}}

\DeclareMathOperator*{\Tot}{Tot}

\DeclareMathOperator{\Spec}{Spec}

\newcommand{\we}{\simeq}
\newcommand{\iso}{\cong}

\theoremstyle{plain}
\newtheorem{theorem}{Theorem}[section]
\newtheorem*{theorem*}{Theorem}
\newtheorem{lemma}[theorem]{Lemma}

\newtheorem{proposition}[theorem]{Proposition}

\newtheorem{corollary}[theorem]{Corollary}
\newtheorem*{corollary*}{Corollary}

\theoremstyle{plain}
\newcounter{zaehler}

\theoremstyle{plain}

\theoremstyle{definition}

\newtheoremstyle{named}{}{}{\itshape}{}{\bfseries}{.}{.5em}{#1 \thmnote{#3}}
\theoremstyle{named}

\theoremstyle{definition}
\newtheorem{definition}[theorem]{Definition}

\newtheorem{variant}[theorem]{Variant}

\newtheorem{notation}[theorem]{Notation}

\newtheorem{example}[theorem]{Example}
\newtheorem*{example*}{Example}

\newtheorem{question}[theorem]{Question}
\newtheorem*{question*}{Question}

\newtheorem{remark}[theorem]{Remark}

\begin{document}

\maketitle
\begin{abstract}
    \noindent
    We introduce the notion of $\bE_\infty$-descendability as well as a
    derived variant. We prove that several classes of descendable maps of
    commutative rings are $\bE_\infty$-descendable. As an application, we prove a variant of
    Tannaka duality.
\end{abstract}

\section{Introduction}

Let $f\colon R\rightarrow S$ be a morphism of $\bE_\infty$-rings with $$S^\bullet\colon S\stack{3}
S\otimes_RS\stack{5}\cdots$$ the associated
\v{C}ech complex, a cosimplicial $\bE_\infty$-ring. There is an associated $\Tot$-tower
$\{\Tot^n(S^\bullet)\}_n$ with limit $\Tot(S^\bullet)=\lim_n\Tot^n(S^\bullet)$.
In~\cite{mathew-galois}, Mathew declares $f$ to be descendable if the pro-object
$\{\Tot^n(S^\bullet)\}_n$ is equivalent
to the constant pro-object $\{R\}$ in $\Pro(\D(R))$, where $\D(R)$ denotes the $\infty$-category of
$R$-modules and $\Pro(\D(R))$ is its pro-category.

A consequence of descendability is that the adjunction $\D(R)\rightleftarrows\D(S)$ is comonadic.
More generally, if $\Cscr$ is any $R$-linear $\infty$-category, then the adjunction
$\Cscr\rightleftarrows\Mod_S(\Cscr)$ is comonadic, and the converse is true. Mathew proved
in~\cite{mathew-galois} that if $R\rightarrow S$ is a faithfully flat map of commutative rings and
if $R$ has cardinality at most $\aleph_n$ or $S$ is $\aleph_n$-compactly generated over $R$ for some $n$, then $R\rightarrow S$ is
descendable. Non-descendable faithfully flat maps have recently been constructed by
Aoki~\cite{aoki-descendable} and Zelich~\cite{zelich,zelich-indivisible}. Aoki's examples are Boolean rings of cardinality $\aleph_\omega$.

Other examples of descendable maps in algebraic geometry are quotients by nilpotent
ideals~\cite{mathew-galois} and $h$-covers of Noetherian schemes~\cite{bhatt-scholze-affine}.

Descendability is equivalent to the natural map $R\rightarrow\Tot^n(S^\bullet)$ admitting an
$R$-module retraction for some $n\geq 0$. For example, if $R$ is a regular, Noetherian commutative
ring and if $R\subseteq S$ is a finite extension with $S$ commutative, then the direct summand
conjecture (a theorem of Andr\'e, see~\cite{andre-summand,bhatt-summand}) implies that $R\rightarrow S$ admits a retraction as an
$R$-module, so one can take $n=0$ above.

\medskip
In this paper, we introduce the notion of $\bE_\infty$-descendability. A map $f\colon R\rightarrow
S$ is $\bE_\infty$-descendable if the pro-object $\{R\}\rightarrow\{\Tot^n(S^\bullet)\}_n$ is an
equivalence in $\Pro(\CAlg(\D(R)))$, the pro-category of
$\bE_\infty$-$R$-algebras. The forgetful functor $\Pro(\CAlg(\D(R)))\rightarrow\Pro(\D(R))$ is
typically not
conservative (see Example~\ref{ex:not_conservative}), so the condition of $\bE_\infty$-descendability is {\em a priori} stronger than that
of descendability. As
many $\bE_\infty$-rings of interest in algebraic geometry are in fact canonically derived
commutative rings in the sense of~\cite{raksit}, we also study the related notion of
$\DAlg$-descendability, which implies $\bE_\infty$-descendability.
Over $\bQ$ the notions of $\bE_\infty$ and $\DAlg$-descendability agree.

A simple consequence of $\bE_\infty$-descendability of $R\rightarrow S$ is that if
$F\colon\CAlg(\D(R))\rightarrow\Cscr$ is any functor, then
$\{F(R)\}\rightarrow\{F(\Tot^n(S^\bullet))\}_n$ is an equivalence in $\Pro(\Cscr)$.
If one requires only descendability, then the corresponding result only works for functors $F$ which
factor through the forgetful functor $\CAlg(\D(R))\rightarrow\D(R)$.

We show that many descendable morphisms are $\DAlg$-descendable. As observed
in~\cite[Thm.~3.13]{mathew-examples}, where it is credited to Srikanth Iyengar, faithfully flat extensions $R\rightarrow S$ are descendable if $R$ is
Noetherian of finite Krull dimension. We extend this result to $\bE_\infty$-descendability in
characteristic $0$.

\begin{theorem}[Theorem~\ref{thm:noetherian}]\label{thm:noetherian_intro}
    If $R$ is a Noetherian $\bQ$-algebra of finite Krull dimension, then every faithfully
    flat map $R\rightarrow S$ is $\bE_\infty$-descendable.
\end{theorem}

As a result, fppf maps $R\rightarrow S$ of $\bQ$-algebras are descendable.

Bhatt and Scholze prove in~\cite[Prop.~11.25]{bhatt-scholze-affine} that $h$-covers of Noetherian
schemes induce descendable morphisms. We show in some low-dimensional cases that the stronger conclusion
of $\DAlg$-descendability holds.

\begin{proposition}[Proposition~\ref{prop:h}]
    Let $g\colon R\rightarrow S$ be an integral extension of Noetherian rings. If
    $\dim R\leq 1$, then $f$ is $\DAlg$-descendable. If $\dim R\leq 2$ and $f$ is
    birational, then $f$ is $\DAlg$-descendable.
\end{proposition}

In characteristic $p$, we have two flavors of results. The first establishes $\DAlg$-descendability
of Frobenius for smooth algebras over perfect $\bF_p$-algebras.

\begin{proposition}[Proposition~\ref{prop:frobenius}]
    Let $k$ be a perfect commutative $\bF_p$-algebra and let $R$ be a smooth commutative $k$-algebra. The Frobenius
    $\varphi\colon R\rightarrow R$ is $\DAlg$-descendable.
\end{proposition}

The second is about faithfully flat extensions of $p$-Boolean rings.
Say that a commutative $\bF_p$-algebra $R$ is $p$-Boolean if $x^p=x$ for all $x\in R$.
There is a notion of $p$-Boolean derived commutative $\bF_p$-algebra as in~\cite{antieau-spherical}. These form an
$\infty$-category $\DAlg_{\bF_q}^{\varphi=1}$ and there is a corresponding notion of
$p$-Boolean $\DAlg$-descendability, which implies the others.

\begin{theorem}[Corollary~\ref{cor:boolean_descendable}]\label{thm:boolean_descendable_intro}
    Let $k$ be a $p$-Boolean ring and let $R$ be a faithfully flat $p$-Boolean $k$-algebra.
    Then $k\rightarrow R$ is $p$-Boolean $\DAlg$-descendable if
    \begin{enumerate}
        \item[{\em (i)}] $k$ has cardinality at most $\aleph_n$ or $k\rightarrow R$ is
            $\aleph_n$-compact for some natural number $n$, or
        \item[{\em (ii)}] $\Spec k$ is extremally disconnected.
    \end{enumerate}
\end{theorem}

Note that the situations of Theorem~\ref{thm:boolean_descendable_intro} are precisely those where
we know that $k\rightarrow R$ is descendable. Moreover, Aoki's example from~\cite{aoki-descendable} implies
that there are non-descendable (and hence non-$p$-Boolean $\DAlg$-descendable) faithfully flat maps of
$p$-Boolean rings.

As an application of $\bE_\infty$-descendability we prove a new Tannaka duality theorem.
Recall that in previous work~\cite{bhatt_tannaka,bhatt-halpern-leistner,dag8,stefanich-tannaka}, mapping anima $\Hom(S,X)$ are shown to be equivalent under various
conditions on $S$ and $X$ to anima of
symmetric monoidal left adjoint functors $\QCoh(X)\rightarrow\QCoh(S)$ which preserve connective objects.
We give a related result in the context of functors between categories of commutative algebras for
certain stacks in spectral algebraic geometry.

\begin{theorem}[Theorem~\ref{thm:tannaka}]\label{thm:tannaka_intro}
    Let $X$ be a quasicompact geometric stack with affine diagonal, and assume that there exists an
    fpqc surjection $p\colon U \rightarrow X$ with $U$ affine such that the commutative algebra
    $p_*(\mathcal{O}_U)$ in $\QCoh(X)$ is $\bE_\infty$-descendable. Then, for every affine scheme
    $S$, the morphism of anima
    \[
    \Hom(S, X) \rightarrow \Hom_{\Pr^L_{\CAlg(\D(\bS))/}}( \CAlg(\QCoh(X)) , \CAlg(\QCoh(S)))
    \]
    given by sending $f$ to $f^*$ is an embedding, and its image consists of those functors which preserve finite limits and connective objects.
\end{theorem}

The conditions of the theorem can be verified in various cases: for instance, they hold for quasicompact semiseparated schemes, and for quotients of such schemes by actions of affine algebraic groups in characteristic zero. We note that the  weaker assumption of descendability of $p_*(\Oscr_U)$ was recently studied by Jiang in~\cite{jiang_descendability}.

\medskip

Our methods are fundamentally via obstruction theory. For example, in the settings of
Theorems~\ref{thm:noetherian_intro} and~\ref{thm:boolean_descendable_intro}, we
show that for suitable $n\geq 1$ the map
$R\rightarrow\Tot^n(S^\bullet)$ admits a retraction by inductively killing negative homotopy groups
of $\Tot^n(S^\bullet)$
without altering $\pi_0$. As $R\iso\pi_0\Tot^n(S^\bullet)$ for $n\geq 1$, this proves
$\bE_\infty$-descendability. There are two subtleties. One is that in order to be able to kill
negative homotopy groups in a controlled way we require certain extension classes to vanish. This
is the source of the homological requirements in the results above. The second is that we need to
control free algebras so as to be able maintain an understanding of the results of killing classes.
In characteristic $0$ or in the $p$-Boolean cases, the cohomology of these free algebras is closely
related to the cohomology of Eilenberg--Mac Lane anima (by rational or $p$-adic homotopy theory;
see~\cite{antieau-spherical}).

As a consequence of our methods, we establish a generalization of the famous Gleason theorem, which
identifies the extremally disconnected profinite sets as the projective objects in the category of
compact Hausdorff spaces.
Say that an anima $X$ is $p$-finite if $\pi_0X$ is finite, $X$ is truncated, and each homotopy group
$\pi_i(X,x)$ for $i\geq 1$ is a finite $p$-group. Let $\Sscr_{p\fin}\subseteq\Sscr$ be the full
subcategory of $p$-finite anima and let $\Pro(\Sscr_{p\fin})$ be its pro-category. There is an
equivalence $\Pro(\Sscr_{p\fin})^\op\we\DAlg_{\bF_p}^{\varphi=1}$ which extends the classical Stone
duality between $\Pro(\Fin)^\op$ and $p$-Boolean commutative rings; see~\cite{antieau-spherical}.

\begin{theorem}[Generalized Gleason theorem, Theorem~\ref{thm:enough_projectives}]
    Suppose that $k$ is a $p$-Boolean ring and that $\Spec k$ is extremally disconnected. If
    $k\rightarrow R$ is a map of $p$-Boolean derived rings where $k\rightarrow\pi_0R$ is injective,
    then there is a retraction $R\rightarrow k$.
\end{theorem}

In the language of pro-$p$-finite anima, $\Spec R\rightarrow\Spec k$ has a section.
We can make $\Pro(\Sscr_{p\fin})$ into a site using the coherent topology, whose covers are given
by finite jointly effectively epimorphic families of maps.
As we shall see, the coherent topology on
$\Pro(\Sscr_{p\fin})$ agrees with the topology whose covers are finite families $\{T_i\rightarrow
S\}$ such that the induced map $\Oscr(S)\rightarrow\prod_i\Oscr(T_i)$ on continuous cochains with
$\bF_p$-coefficients is
ccff in the sense of~\cite{mathew-mondal}. The generalized Gleason theorem implies the following result.

\begin{theorem}[Corollary~\ref{cor:condensed}]
    \label{thm:intro_condensed}
    The $\infty$-category of hypersheaves of anima on
    $\Pro(\Sscr_{p\fin})$ with respect to the coherent topology agrees with the $\infty$-category of
    condensed anima.
\end{theorem}

The variant of Theorem~\ref{thm:intro_condensed} in which $p$-finite anima are
replaced by $\pi$-finite anima has recently is established by Peter Haine in forthcoming work using different methods.

Affine stacks were introduced by To\"en~\cite{toen-affines}. These are precisely the fpqc
hypersheaves which are represented by coconnective derived commutative rings;
see~\cite{mathew-mondal}. We also use the generalized Gleason theorem to prove the next corollary.

\begin{corollary}[Theorem~\ref{thm:points}]\label{cor:intro_affine}
    Let $\Xscr$ be an affine stack over a finite field $\bF_p$. If
    $\pi_0\R\Gamma(\Xscr,\Oscr)\iso\bF_p$, then $\Xscr$ admits a rational point.
\end{corollary}

In fact, $\Xscr$ admits a rational point inside its ``spatial locus'' $\Xscr^{\varphi=1}\rightarrow\Xscr$.

As a particular instance of Corollary~\ref{cor:intro_affine}, we see that if $X$ is a smooth proper
geometrically connected scheme over a field $\bF_p$, then its unipotent homotopy type
$U(X)=\Spec\R\Gamma(X,\Oscr)$,
in the sense of Mondal--Reinecke~\cite{mondal-reinecke}, admits a rational point.

Dmitry Kubrak has pointed to us that Corollary~\ref{cor:intro_affine} can be deduced
in a different way
from the proof of~\cite[Thm.~2.4.5]{toen-affines}, which in particular shows that it holds over any
perfect field of characteristic $p$.

We consider this exploration of $\bE_\infty$-descendability to be only a first step and propose the
following open problems as subjects for further study.
\begin{enumerate}
    \item[(a)] Is there an example of a descendable map $f\colon R\rightarrow S$ of
        $\bE_\infty$-rings (or commutative rings, etc.) which is not
        $\bE_\infty$-descendable?
    \item[(b)] Specifically, as $\KO\rightarrow\KU$ is descendable by~\cite{mathew-galois}, is it
        $\bE_\infty$-descendable?
    \item[(c)] Similarly, the map $\bF_p^{\B\bZ/p^k}\rightarrow\bF_p^{\B\bZ/p}$ induced by
        $\bZ/p\subseteq\bZ/p^k$ is descendable for $k\geq 1$ by~\cite[Cor.~4.10]{mathew-galois}. Is
        it $\bE_\infty$-descendable or $\DAlg$-descendable?
    \item[(d)] Are faithfully flat maps of finite type $\bF_p$-algebras $\DAlg$-descendable?
    \item[(e)] Is $\bF_p[x]\rightarrow\bF_p[x^{1/p^\infty}]$ $\DAlg$-descendable?
\end{enumerate}

\paragraph{Outline.}
In Section~\ref{sec:descendability} we introduce $\bE_\infty$ and $\DAlg$-descendability, prove the
basic properties, and give the first examples. Section~\ref{sec:frobenius} contains the argument
that the Frobenius on smooth algebras is $\DAlg$-descendable. In Sections~\ref{sec:ccff}
and~\ref{sec:boolean} we introduce coconnectively faithfully flat maps after~\cite{mathew-mondal}
and use these to prove that many faithfully flat maps of $p$-Boolean rings are
$\DAlg^{\varphi=1}$-descendable. We turn to $\bQ$-algebras in Section~\ref{sec:noetherian}, showing that in characteristic $0$ many
faithfully flat maps are $\bE_\infty$-descendable. In Section~\ref{sec:finite}, we show that at least some $h$-covers are.
Finally, we give our results on Tannaka duality in Section~\ref{sec:tannaka}.

\paragraph{Acknowledgments.}
We thank Ko Aoki, Robert Burklund, Peter Haine, Dmitry Kubrak, Deven Manam, Akhil Mathew, Shubhodip Mondal, and Peter Scholze for
helpful conversations about Boolean rings and descent.

BA was supported by NSF grant DMS-2152235, by Simons Fellowship 00005925, by the Simons
Collaboration on Perfection, and by the hospitality of UC Berkeley. The authors thank the Max Planck Institute for Mathematics in Bonn,
where this work was largely carried out.

\section{$\bE_\infty$-descendability}\label{sec:descendability}

Let $f\colon R\rightarrow S$ be a morphism of $\bE_\infty$-rings. In~\cite{mathew-galois}, Akhil Mathew defines $f$ to be
descendable if the induced map $\{R\}\rightarrow\{\Tot^n(S^\bullet)\}_n$ is an equivalence of
pro-objects of $\D(R)$, the $\infty$-category of $R$-module spectra. Here, $S^\bullet$ is the
\v{C}ech complex of $f$, which is a cosimplicial $\bE_\infty$-ring.

\begin{definition}[$\bE_\infty$-descendability]
    Let $f\colon R\rightarrow S$ be a map of $\bE_\infty$-rings and let
    $S^\bullet$ be the \v{C}ech nerve.
    Say that $f$ is $\bE_\infty$-descendable if the
    map of towers $\{R\}\rightarrow\{\Tot^n(S^\bullet)\}_n$ is a
    pro-equivalence of pro-$\bE_\infty$-rings.
\end{definition}

It makes sense in fact to speak of $\bE_\infty$-descendable morphisms in general presentably
symmetric monoidal stable $\infty$-categories. We also use the following variant.

\begin{variant}[$\DAlg$-descendability]
    If $f\colon R\rightarrow S$ is a map of derived commutative rings in the sense of~\cite{raksit}, then we
    say that $f$ is $\DAlg$-descendable if the map of pro-towers
    $\{R\}\rightarrow\{\Tot^n(S^\bullet)\}_n$ is a pro-equivalence of
    pro-derived commutative rings. 
\end{variant}

If $f$ is $\bE_\infty$-descendable, then it is descendable,
and $\DAlg$-descendability implies $\bE_\infty$-descendability.

In the remainder of this section, several results are stated for $\bE_\infty$-descendability.
Their $\DAlg$-descendable analogues hold as well. This applies to
Proposition~\ref{prop:equivalent}, Lemma~\ref{lem:two_three}, Lemma~\ref{lem:push_forwards},
Lemma~\ref{lem:base_change}, and Proposition~\ref{prop:eoo_nilpotence}.

Here is the $\bE_\infty$-analogue of~\cite[Prop.~3.20]{mathew-galois}.

\begin{proposition}\label{prop:equivalent}
    Let $f\colon R\rightarrow S$ be a map of $\bE_\infty$-rings. Let
    $S^\bullet$ be the \v{C}ech complex. The following hypotheses are equivalent.
    \begin{enumerate}
        \item[{\em (i)}] The morphism $f$ is $\bE_\infty$-descendable.
        \item[{\em (ii)}] The map $R\rightarrow\Tot^n(S^\bullet)$ admits an
            $\bE_\infty$-retraction for some $n\geq 0$.
        \item[{\em (iii)}] If $\Cscr$ is the smallest full subcategory of $\CAlg(\D(R))$
            which contains the $\bE_\infty$-algebras which admit a map from $S$ and which is closed under finite limits and
            retractions, then $\Cscr$ contains $R$.
    \end{enumerate}
\end{proposition}

\begin{proof}
    An inverse of pro-systems gives that (i) implies (ii). As
    $\Tot^n(S^\bullet)$ is a finite limit of tensor powers of $S$, (ii) implies
    (iii). Now, assume (iii). Consider the full subcategory $\Dscr$ of
    $\CAlg(\D(R))$ on the objects $T$ such that
    $\{T\}\rightarrow\{\Tot^n(T\otimes_RS^\bullet)\}_n$ is a pro-isomorphism of
    $\bE_\infty$-$R$-algebras. We claim that $\Cscr$ is contained in $\Dscr$.
    This will imply (iii) implies (i) in light of the assumption that
    $R\in\Cscr$.

    As $S\otimes_RS^\bullet$ extends to a split cosimplicial object with limit $S$, it follows that
    $S$ is contained in $\Dscr$ and so is any
    commutative $S$-algebra (see~\cite[Ex.~3.11]{mathew-galois}).
    As pro-constant objects are closed under finite limits and retracts, $\Cscr\subseteq\Dscr$.
\end{proof}

\begin{lemma}\label{lem:two_three}
    Suppose that $R\xrightarrow{f}S\xrightarrow{g}T$ are maps of $\bE_\infty$-algebras.
    \begin{enumerate}
        \item[{\em (a)}] If $f$ and $g$ are $\bE_\infty$-descendable, then so is $g\circ f$.
        \item[{\em (b)}] If $g\circ f$ is $\bE_\infty$-descendable, then so is $f$.
    \end{enumerate}
\end{lemma}

\begin{proof}
    Let $\Cscr_T$ be the smallest full subcategory of $\CAlg(\D(R))$ consisting of objects which admit a
    map from $T$ and which is closed under finite limits and retracts; define $\Cscr_S$ similarly. By
    $\bE_\infty$-descendability of $S\rightarrow T$, we have that $S\in\Cscr_T$ and indeed that any
    $S'\in\CAlg(\D(R))$ which admits a map from $S$ is in $\Cscr_T$. Since $\Cscr_T$ is closed under
    finite limits and retractions, it follows that
    $\Cscr_S\subseteq\Cscr_T$. As $R\in\Cscr_S$, we have that $R\in\Cscr_T$ and hence $R\rightarrow
    T$ is $\bE_\infty$-descendable by Proposition~\ref{prop:equivalent}.

    For (b), choose $n$ such that there is an $R$-algebra map $\Tot^n(T^\bullet)\rightarrow R$ such
    that the composition $R\rightarrow\Tot^n(T^\bullet)\rightarrow R$ is equivalent to the identity of $R$.
    Then, composing with $\Tot^n(S^\bullet)\rightarrow\Tot^n(T^\bullet)$, we see that
    $R\rightarrow S$ is $\bE_\infty$-descendable by Proposition~\ref{prop:equivalent}.
\end{proof}

\begin{lemma}\label{lem:push_forwards}
    Let $F\colon\Cscr\rightarrow\Dscr$ be an exact lax symmetric monoidal functor of
    symmetric monoidal stable $\infty$-categories. If $A\rightarrow B$ is $\bE_\infty$-descendable
    in $\Cscr$,
    then $F(A)\rightarrow F(B)$ is $\bE_\infty$-descendable.
\end{lemma}

\begin{proof}
    We have $F(\Tot^n(B^\bullet))\we \Tot^n(F(B^\bullet))$ as $F$ commutes with finite limits.
    There is a natural map $F(B)^\bullet\rightarrow F(B^\bullet)$ of cosimplicial
    $\bE_\infty$-algebras in $\Dscr$ and hence a map
    $\Tot^n(F(B)^\bullet)\rightarrow F(\Tot^n(B^\bullet))\rightarrow F(A)$ whose composition with
    $F(A)\rightarrow\Tot^n(F(B)^\bullet)$ is the identity. This establishes
    $\bE_\infty$-descendability of $F(A)\rightarrow F(B)$.
\end{proof}

\begin{variant}
    If $\Cscr$ and $\Dscr$ admit notions of derived commutative ring and if
    $F$ preserves these, then Lemma~\ref{lem:push_forwards} applies in the
    $\DAlg$-descendable case too.
\end{variant}

\begin{lemma}\label{lem:base_change}
    If $R\rightarrow S$ is $\bE_\infty$-descendable and if $R\rightarrow R'$ is an arbitrary
    morphism of $\bE_\infty$-rings, then the induced morphism $R'\rightarrow S'$ is
    $\bE_\infty$-descendable, where $S'=S\otimes_RR'$.
\end{lemma}

\begin{proof}
    This follows from Lemma~\ref{lem:push_forwards}.
\end{proof}

Now we give our first examples of $\bE_\infty$ and $\DAlg$-descendable morphisms.

\begin{example}
    If $f\colon R\rightarrow S$ admits an $\bE_\infty$-retraction, then $f$ is
    $\bE_\infty$-descendable.
\end{example}

\begin{corollary}\label{cor:zariski}
    If $R\rightarrow\prod_{i=1}^nR[1/f_i]$ is a Zariski cover, then it is
    $\DAlg$-descendable.
\end{corollary}

\begin{proof}
    Indeed, $R$ is a finite limit of $\prod_{i=1}^nR[1/f_i]$-algebras.
\end{proof}

\begin{corollary}\label{cor:pullback}
    If $$\xymatrix{
        R\ar[r]\ar[d]&S_0\ar[d]\\
        S_1\ar[r]&T
    }$$ is a pullback of $\bE_\infty$-rings (or derived commutative rings), then $R\rightarrow S_0\times S_1$
    is $\bE_\infty$-descendable (or $\DAlg$-descendable).
\end{corollary}

\begin{proof}
    The $\bE_\infty$-rings $S_0$, $S_1$, and $T$ all admit maps from $S_0\times
    S_1$, so the result follows by the criterion of Proposition~\ref{prop:equivalent}(iii).
\end{proof}

\begin{example}
    The corollary applies to arithmetic fracture squares, such as
    $$\xymatrix{
        \bZ\ar[r]\ar[d]&\bQ\ar[d]\\
        \bA_f\ar[r]&(\bA_f)_\bQ,
    }$$ where $\bA_f$ is the ring of finite ad\`eles.
    Thus, $\bZ\rightarrow\bQ\times\bA_f$ is $\DAlg$-descendable and hence $\bE_\infty$-descendable.
\end{example}

\begin{example}
    If $I$ is a square-zero
    ideal in an $\bE_\infty$-ring $R$ (or derived commutative ring) with quotient $R/I=\cofib(I\rightarrow R)$, then $R$ fits into a pullback
    $$\xymatrix{
        R\ar[r]\ar[d]&R/I\ar[d]^\d\\
        R/I\ar[r]^\can & R/I\oplus I[1]
    }$$
    of $\bE_\infty$-rings (or derived commutative rings),
    where the bottom horizontal arrow is the canonical inclusion into the
    trivial square-zero extension and the left vertical arrow defines the
    derivation classifying the square-zero extension. In this case, the two maps $R\rightarrow R/I$
    are equivalent so that by, Corollary~\ref{cor:pullback}, $R\rightarrow R/I$ is
    $\bE_\infty$-descendable (or $\DAlg$-descendable).
    In particular, if $R$ is a connective $\bE_\infty$-ring with only finitely many non-zero
    homotopy groups, then $R\rightarrow\pi_0R$ is $\bE_\infty$-descendable.
\end{example}

\begin{example}
    If $X$ is a connected compact anima and $R$ is a derived commutative ring, then $R^X\rightarrow R$ is $\DAlg$-descendable
    for any derived commutative ring $R$; similarly, if $R$ is $\bE_\infty$,
    then $R^X\rightarrow R$ is $\bE_\infty$-descendable.
    Indeed, $R^X$ is equivalent to $\lim_X R$ in $\CAlg_{R^X}$ and this limit is a retract of a
    finite limit.
\end{example}

Mathew gives a criterion in~\cite[Prop.~3.27]{mathew-galois}
for when a morphism $f\colon R\rightarrow S$ is descendable in terms of
the degree of nilpotence of the $S$-null morphisms in $\D(R)$. Specifically, if
the tensor-ideal $\Iscr_S$ of morphisms in $\D(R)$ which become nullhomotopic after tensoring
with $S$ (the $S$-null morphisms) has the property
that $\Iscr_S^n$ vanishes for some positive integer $n$, then $R\rightarrow S$
is descendable. In particular, the composite of $n$ maps which are $S$-null is
in fact nullhomotopic in $\D(R)$. There is an analogue in our setting. To describe it, we recall
the notion of a Smith ideal, which we will call simply an ideal. Classical references are~\cite{hovey-smith,bruner-isaksen}.

Let $\Delta^{1,\op,\vee}$ be the category $\{1\rightarrow 0\}$ with the symmetric monoidal structure
given by taking maxima. Similarly, let $\Delta^{1,\op,\wedge}$ denote the same category but with the
symmetric monoidal structure given by taking minima. Day convolution, in the sense
of~\cite{glasman} or~\cite[Sec.~2.2.6]{ha}, induces two symmetric monoidal
structures $\otimes_\max$ and $\otimes_\min$ on the functor $\infty$-category $\Fun(\Delta^{1,\op},\D(R))$.
It is not difficult to see that $\CAlg(\Fun(\Delta^{1,\op},\D(R))^{\otimes_\min}$ is equivalent to
$\CAlg(\D(R))^{\Delta^{1,\op}})$, the $\infty$-category of morphisms of
$\bE_\infty$-$R$-algebras. Moreover, taking cofibers induces a symmetric monoidal equivalence
$$\cofib\colon\Fun(\Delta^{1,\op},\D(R))^{\otimes_\max}\rightarrow\Fun(\Delta^{1,\op},\D(R))^{\otimes_\min}$$
and hence an equivalence
$$\cofib\colon\CAlg(\Fun(\Delta^{1,\op},\D(R))^{\otimes_\max})\rightarrow\CAlg(\Fun(\Delta^{1,\op},\D(R))^{\otimes_\min}).$$
An $\bE_\infty$-ideal, or simply ideal, $I\rightarrow R$ is an object of $\CAlg(\Fun(\Delta^{1,\op},\D(R))^{\otimes_\max})$.
There is similarly a notion of an ideal in a derived commutative ring.
If $R\rightarrow S$ is a morphism of $\bE_\infty$-rings, then by definition $\fib(R\rightarrow
S)\rightarrow R$ is an ideal. Conversely, if $I\rightarrow R$ is an ideal, then
$R\rightarrow\cofib(I\rightarrow R)$ admits the canonical structure of an $\bE_\infty$-$R$-algebra.

\begin{proposition}[$\bE_\infty$-nilpotence and
    $\bE_\infty$-descendability]\label{prop:eoo_nilpotence}
    Let $R\rightarrow S$ be a morphism of $\bE_\infty$-rings and let
    $I=\fib(R\rightarrow S)$. The following are equivalent:
    \begin{enumerate}
        \item[{\em (a)}] $R\rightarrow S$ is $\bE_\infty$-descendable;
        \item[{\em (b)}] the ideal $I^{\otimes n+1}\rightarrow R$ admits a map of
            ideals to $0\rightarrow R$ for some $n\geq 0$.
    \end{enumerate}
\end{proposition}

\begin{proof}
    If $R\rightarrow S$ is $\bE_\infty$-descendable, then there is a retraction of
    $\bE_\infty$-rings
    $R\rightarrow \Tot^{n}(S^\bullet)\rightarrow R$ for some $n\geq 0$.
    Taking the fibers of $R\rightarrow\Tot^{n}(S^\bullet)$ and $R\rightarrow
    R$, we obtain a morphism of ideals
    $$\xymatrix{
        I^{\otimes n+1}\ar[r]\ar[d]&R\ar@{=}[d]\\
        0\ar[r]&R,
    }$$ where we use~\cite[Prop.~2.14]{MNN17} to identify the fiber of
    $R\rightarrow\Tot^{n}(S^\bullet)$ with $I^{\otimes n+1}$ for $n\geq 0$.
    Thus, (a) implies (b). Conversely, (b) is the statement that such a square exists; taking
    horizontal cofibers yields (a).
\end{proof}

Much of this paper is dedicated to proving that known examples of descendable morphisms are
$\bE_\infty$-descendable. We do not know of a single example of a descendable morphism which is not
$\bE_\infty$-descendable. In the remainder of this section we note that there does not appear to be
a reason why that descendability implies the stronger condition of
$\bE_\infty$-descendability.

Should the forgetful functor
$\Pro(\CAlg(\D(R))) \rightarrow \Pro(\D(R))$ be conservative, descendable morphisms would be
$\bE_\infty$-descendable. This is however never the case (unless $R = 0$), as the following example shows.

 \begin{example}\label{ex:not_conservative}
 Let $R$ be a nonzero $\bE_\infty$-ring, and consider the $\bE_\infty$ $R$-algebra of dual numbers $R[\varepsilon]$
     (in other words, this is the split square zero extension of $R$ by $R$). We have a commutative square of $\infty$-categories
 \[
 \begin{tikzcd}
     \CAlg(\D({R[\varepsilon]})) \arrow{d}{} \arrow{r}{} & \D(R[\varepsilon]) \arrow{d}{} \\
     \CAlg(\D(R)) \arrow{r}{} & \D(R)
 \end{tikzcd} 
 \]
 where the maps are the evident forgetful functors. Passing to $\infty$-categories of pro-objects we obtain a commutative square
  \[
 \begin{tikzcd}
     \Pro(\CAlg(\D(R[\varepsilon]))) \arrow{d}{} \arrow{r}{} & \Pro(\D({R[\varepsilon]})) \arrow{d}{} \\
     \Pro(\CAlg(\D(R))) \arrow{r}{} & \Pro(\D(R)).
 \end{tikzcd} 
 \]
 The right vertical arrow in the above square is not conservative: for instance, the pro-$R[\varepsilon]$-module
 \[
 R \leftarrow R[-1] \leftarrow R[-2] \leftarrow \cdots 
 \]
 obtained by iterated composition of the canonical class in $\operatorname{Ext}^1_{R[\varepsilon]}(R)$ is nonzero, but its underlying pro-$R$-module is zero.
 
Passing to split square zero extensions, we deduce that the diagonal map in the above square is not conservative: the pro-$R[\varepsilon]$-algebra
\[
R[\varepsilon] \oplus R \leftarrow R[\varepsilon] \oplus R[-1] \leftarrow R[\varepsilon] \oplus R[-2] \leftarrow \cdots
\]
admits a map to $R[\varepsilon]$ which is not an isomorphism, but which becomes an isomorphism in $\Pro(\D(R))$.

Observe now that the left vertical arrow in the above square is conservative.
     Indeed, we have an equivalence $\CAlg(\D({R[\varepsilon]})) \we \CAlg(\D(R))_{R[\varepsilon]/}$,
     and hence $\Pro(\CAlg(\D({R[\varepsilon]}))) \we \Pro(\CAlg(\D(R)))_{R[\varepsilon] /}$. It
     follows that the functor $\Pro(\CAlg(\D(R))) \rightarrow \Pro(\D(R))$ is not conservative. Concretely, the pro-$R$-algebra 
\[
R[\varepsilon] \oplus R \leftarrow R[\varepsilon] \oplus R[-1] \leftarrow R[\varepsilon] \oplus R[-2] \leftarrow \cdots
\]
admits a map to $R[\varepsilon]$ which is not an isomorphism of pro-$R$-algebras but which is an isomorphism of pro-$R$-modules.
 \end{example}

Even in the case when $R$ is a (discrete) commutative ring, the counterexample to conservativity of
the forgetful functor $\Pro(\CAlg(\D(R))) \rightarrow \Pro(\D(R))$ requires one to use non-discrete
$\bE_\infty$-algebras. It turns out that in a truncated context this forgetful functor is in fact
conservative. We record this fact for future reference; it plays no role in the rest of this paper.

\begin{proposition}\label{prop:ff_when_truncated}
Let $0 \leq n < \infty$ and let $\Cscr$ be a symmetric monoidal $(n,1)$-category. Equip $\Pro(\Cscr)$ with its canonical symmetric monoidal structure (so that the tensor product preserves cofiltered limits in each variable and the inclusion $\Cscr \rightarrow \Pro(\Cscr)$ is symmetric monoidal). Then the functor $\Pro(\CAlg(\Cscr)) \rightarrow \CAlg(\Pro(\Cscr))$ obtained by pro-extension of the inclusion $\CAlg(\Cscr) \rightarrow \CAlg(\Pro(\Cscr))$ is fully faithful. 
\end{proposition}

\begin{corollary}
Let $0 \leq n < \infty$ and let $\Cscr$ be a symmetric monoidal $(n,1)$-category. Then the functor $\Pro(\CAlg(\Cscr))  \rightarrow \Pro(\Cscr)$ obtained by pro-extension of the forgetful functor $\CAlg(\Cscr) \rightarrow \Cscr$ is conservative.
\end{corollary} 
\begin{proof}
The functor in the statement factors as the composition
\[
\Pro(\CAlg(\Cscr)) \rightarrow \CAlg(\Pro(\Cscr)) \rightarrow \Pro(\Cscr)
\]
where the second map is the forgetful functor (which is conservative). The result now follows from an application of Proposition~\ref{prop:ff_when_truncated}, which guarantees that the first map is fully faithful.
\end{proof} 

\begin{lemma}\label{lem:cocompact sections}
Let $0 \leq n < \infty$. Let $p\colon \mathcal{E} \rightarrow \mathcal{B}$ be a cocartesian fibration of
    $(n,1)$-categories and let $p\colon \mathcal{E}' \rightarrow \mathcal{B}$ be the cocartesian
    fibration obtained from $p$ by passing to pro-categories fiberwise. Let $s$ be a section of $p$. If $\mathcal{B}$ is a finite $(n,1)$-category then $s$ is cocompact as a section of $p'$.
\end{lemma}
\begin{proof}
Assume that $\mathcal{B}$ is the pushout of a span of $(n,1)$-categories $\mathcal{B}_0 \leftarrow \mathcal{B}_{01} \rightarrow \mathcal{B}_1$. To show that $s$ is cocompact it suffices to show that its restrictions to $\mathcal{B}_0$, $\mathcal{B}_1$ and $\mathcal{B}_{01}$ are cocompact (as sections of the corresponding base changes of $p'$). Since $\mathcal{B}$ is assumed to be finite, we may reduce to the case when $\mathcal{B}$ is the $1$-simplex $[1]$. The $(n,1)$-category of sections of $p'$ is then given by a fiber product
\[
    \{\id_{[1]}\} \times_{\Fun([1], [1])} \Fun([1], \mathcal{E}').
\]
To show that $s$ is cocompact it suffices to show that its images in the three $(n,1)$-categories
    above are cocompact. We may in this way reduce to showing that $s$ defines a cocompact arrow in
    $\mathcal{E}'$. This follows from~\cite[Prop.~5.3.5.15]{htt} since in this case $\mathcal{E}' = \Pro(\mathcal{E})$.
\end{proof}

\begin{proof}[Proof of Proposition~\ref{prop:ff_when_truncated}]
It suffices to show that every commutative algebra in $\Cscr$ is cocompact when regarded as a
    commutative algebra in $\Pro(\Cscr)$. Since $\Cscr$ and $\Pro(\Cscr)$ are $(n,1)$-categories,
    it is enough to show that every $\bE_{n+1}$-algebra in $\Cscr$ is cocompact when regarded as an
    $\bE_{n+1}$ algebra in $\Pro(\Cscr)$. Arguing inductively using the Dunn additivity theorem
    (see~\cite[Thm.~5.1.2.2]{ha}) we may
    reduce to showing that every associative algebra in $\Cscr$ is cocompact when regarded as an
    associative algebra in $\Pro(\Cscr)$. By~\cite[Thm.~5.4.3.5]{ha}, it suffices to show that every
    nonunital associative algebra in $\Cscr$ is cocompact when regarded as a nonunital associative
    algebra in $\Pro(\Cscr)$. Applying~\cite[Cor.~4.1.6.17]{ha}, we may reduce to showing that
    every nonunital $\bA_{n+2}$-algebra in $\Cscr$ is cocompact when regarded as a nonunital
    $\bA_{n+2}$-algebra in $\Pro(\Cscr)$. This follows from an application of Lemma
    \ref{lem:cocompact sections}, where $\mathcal{B}$ is the opposite of the category of nonempty
    totally ordered sets of cardinality at most $n+2$ and order preserving injections, and $p$ is
    the cocartesian fibration giving the canonical nonunital $\bA_{n+2}$-monoidal structure on
    $\Cscr$.
\end{proof}

\section{Descendability and Frobenius}\label{sec:frobenius}

Our first non-trivial example of $\DAlg$-descendability is a consequence of splitting
the conjugate filtration multiplicatively. The setting for this multiplicative splitting is that of
Raksit's thesis~\cite{raksit}. There, Raksit introduces two notions of filtered derived commutative
rings, one which captures the HKR filtration on Hochschild homology (the derived
commutative rings with respect to the derived algebraic context given by the neutral $t$-structure
on filtered complexes) and one which captures the
Hodge filtration on Hodge-complete derived de Rham cohomology (see~\cite[Def.~5.1.10]{raksit}). These are called the infinitesimal
and crystalline filtered derived commutative rings and are studied further
in forthcoming work of the first named author, where a non-complete version of the crystalline filtered derived
commutative rings is given. Forgetting filtrations produces a derived commutative ring in either
case. 

\begin{proposition}\label{prop:frobenius}
    Let $k$ be a perfect $\bF_p$-algebra and let $R$ be a smooth $k$-algebra. The Frobenius
    $\varphi\colon R\rightarrow R$ is $\DAlg$-descendable.
\end{proposition}

\begin{proof}
    Let $\varphi_{R/k}\colon R^{(1)}\rightarrow R$ be the relative Frobenius. It is enough to show
    that $\varphi_{R/k}$ is $\DAlg$-descendable. However, the relative Frobenius factors as
    $R^{(1)}\iso\H^0(\dR_{R/K})\rightarrow\dR_{R/k}\rightarrow R$. Since $k\rightarrow R$ is smooth,
    $\dR_{R/k}\rightarrow R$ is $\DAlg$-descendable. In fact, $\F^\star_\H\dR_{R/k}\rightarrow R$
    is descendable as a map of filtered derived commutative rings since $\F^n_\H\dR_{R/k}$ vanishes
    for $n$ sufficiently large so we can use the ideal criterion,
    Proposition~\ref{prop:eoo_nilpotence}, or rather its derived commutative version. Thus, it suffices to see
    that $R^{(1)}\rightarrow\dR_{R/k}$ is $\DAlg$-descendable. However, our hypotheses on $R$ imply that
    there is a lift of $R$ to $W_2(k)$ together with a lift of the relative Frobenius. This implies
    that there is a map
    $\LSym^\crys_{R^{(1)}}\Omega^1_{R^{(1)}/k}(1)\rightarrow\F^\star_\H\dR_{R/k}$, where
    $\LSym^\crys_{R^{(1)}}\Omega^1_{R^{(1)}}$ denotes the free crystalline filtered derived
    commutative ring over $R^{(1)}$ on $\Omega^1_{R^{(1)}}$ placed in weight $1$.
    This map is not typically an equivalence of filtered derived commutative rings, but it becomes an equivalence after
    forgetting the filtrations by the Cartier isomorphism~\cite[Thm.~1.2]{deligne-illusie}. We can
    take the zero section of the filtered tangent bundle 
    to produce a filtered map $\LSym^\crys_{R^{(1)}}\Omega^1_{R^{(1)}}(1)\rightarrow R^{(1)}$. In particular,
    $R^{(1)}\rightarrow\dR_{R/k}$ admits a derived commutative $k$-algebra retraction. This
    completes the proof.
\end{proof}

\section{Coconnective faithful flatness}\label{sec:ccff}

We recall the recently described notion of coconnective faithful flatness of
Mathew--Mondal~\cite{mathew-mondal}.

\begin{definition}[Canonical $t$-structure]
    Let $R$ be an $\bE_\infty$-ring. We let $\D(R)_{\geq 0}\subseteq\D(R)$ be
    the full subcategory generated by $R$ under colimits and extensions and we let
    $\D(R)_{\leq 0}\subseteq\D(R)$ be the full subcategory of $R$-modules
    $M$ which are coconnective as spectra: every map of $R$-modules
    $R[i]\rightarrow M$ for $i$ positive is nullhomotopic; equivalently $\pi_iM=0$ for
    $i>0$. These two subcategories define a $t$-structure on $\D(R)$ which is compatible with
    filtered colimits. We refer to it as the canonical $t$-structure.
\end{definition}

\begin{example}
    If $R$ is a connective $\bE_\infty$-ring, then the connective objects in the canonical
    $t$-structure on $\D(R)$ are precisely the $R$-modules $M$ with $\pi_iM=0$ for $i<0$. The heart
    $\D(R)^\heart$ is equivalent to the abelian category $\Mod_{\pi_0R}$ in this case.
\end{example}

\begin{definition}[Bounded above Tor-amplitude]
    Let $R$ be an $\bE_\infty$-ring. An $R$-module $M$ is said to have Tor-amplitude bounded
    above by $a$ if $N\in\D(R)_{\leq 0}$ implies that $M\otimes_R N\in\D(R)_{\leq a}$. The full
    subcategory of $R$-modules with Tor-amplitude bounded above by a fixed integer $a$ is closed
    under extensions, finite limits, and filtered colimits.
\end{definition}

\begin{remark}
    If $R$ is coconnective and $M\in\D(R)$ has Tor-amplitude amplitude bounded above by $a$, then
    $M\in\D(R)_{\leq a}$.
\end{remark}

\begin{remark}
    If $R$ is a connective $\bE_\infty$-ring, then the definition of Tor-amplitude given here
    agrees with the usual definition (to be found for instance in~\cite[Def.~7.2.4.21]{ha}), which is tested only on discrete $R$-modules.
    Moreover, if $R\rightarrow S$ is a map of $\bE_\infty$-rings and $M$ is an $R$-module
    with Tor-amplitude bounded above by $a$, then $M\otimes_RS$ has Tor-amplitude bounded above by $a$.
\end{remark}

\begin{definition}[Coconnective faithful flatness]
    Let $R$ be an $\bE_\infty$-ring. An $R$-module $M$ is coconnectively flat
    (ccf) if $M\otimes_RN$ is coconnective for each coconnective $R$-module
    $M$. In other words, $M$ is ccf if the functor $M\otimes_R(-)\colon
    \D(R)\rightarrow\D(R)$ is left $t$-exact or equivalently if $M$ has Tor-amplitude bounded above
    by $0$. We say that $M$ is coconnectively
    faithfully flat (ccff) if it is ccf and if the functor
    $M\otimes_R(-)\colon\D(R)_{\leq 0}\rightarrow\D(R)_{\leq 0}$ is
    conservative.
\end{definition}

\begin{remark}
    If $R\rightarrow S$ is a map of $\bE_\infty$-rings, it is ccf (or ccff) if
    $S$ is as an $R$-module. Note that $\D(R)\rightarrow\D(S)$ is always right
    $t$-exact, so $R\rightarrow S$ is ccf if and only if
    $\D(R)\rightarrow\D(S)$ is $t$-exact for the canonical $t$-structures.
\end{remark}

Mathew and Mondal give examples of ccff maps of coconnective derived
commutative rings. For example, they show that for every coconnective derived
commutative ring $R$ there is a ccff map $R\rightarrow S$ where $S$ is
a connective derived commutative ring (an animated commutative ring).
However, it appears difficult to construct ccff maps in the world of $\bE_\infty$-ring spectra.

Recall that a commutative ring $A$ is absolutely flat if every $A$-module is flat.
Our first examples of ccff $R$-modules are in the context where $R$ is a coconnective
$\bE_\infty$-ring and $\pi_0R$ is absolutely flat.

\begin{proposition}\label{prop:ccff}
    Let $R$ be a coconnective $\bE_\infty$-ring such that $\pi_0R$ is
    absolutely flat.
    \begin{enumerate}
        \item[{\em (i)}] If $M$ is a coconnective $R$-module, then it is ccf.
        \item[{\em (ii)}] If additionally $\pi_0M$ is faithful as a $\pi_0R$-module, then $M$ is ccff.
    \end{enumerate}
\end{proposition}

\begin{proof}
    Let $\Cscr\subseteq\D(R)_{\leq 0}$ be the full subcategory of
    coconnectively flat $R$-modules. This full subcategory is closed under
    extensions and filtered colimits and also contains all modules of the form
    $M'\otimes_{\pi_0R}R$ for $M'$ a coconnective $\pi_0R$-module, since
    $M'\otimes_{\pi_0R}R\otimes_RN\we M'\otimes_{\pi_0R}N$ and $\pi_0R$ is
    absolutely flat. We claim that there is a functor
    $\bN\rightarrow\Cscr$ whose colimit, computed in $\D(R)$, is $M$; this will prove part (i).
    To see this, let $M_0= M\otimes_{\pi_0R}R\rightarrow M$ be the counit map of the adjunction
    $\D(\pi_0R)\rightleftarrows\D(R)$.
    This map is surjective on homotopy groups, $\pi_0M_0\rightarrow\pi_0M$ is an isomorphism, and $M_0\in\Cscr$.
    Suppose that a diagram
    $M_0\rightarrow\cdots M_i\rightarrow M$ 
    has been constructed for some $i\geq 0$ with the property that
    $M_j\rightarrow M$ is surjective on homotopy for $0\leq j\leq i$ (it
    suffices for this to hold for $j=0$), each $\pi_0M_j\rightarrow\pi_0M$ is an isomorphism, and the kernel of
    $\pi_*M_j\rightarrow\pi_*M$ dies in $\pi_*M_{j+1}$ for $0\leq j<i$.
    Let $F_i=\fib(M_i\rightarrow M)$. We have that $F_i$ is
    $(-1)$-coconnective. Now, consider the composition of the counit map
    $F_i\otimes_{\pi_0R}R\rightarrow F_i$ with $F_i\rightarrow M_i$. Since the
    composition of this map with $M_i\rightarrow M$ is nullhomotopic, there is
    an induced map $\cofib(F_i\otimes_{\pi_0R}R\rightarrow M_i)\rightarrow M$.
    Let $M_i\rightarrow M_{i+1}$ be the cofiber together with its induced
    $2$-cell witnessing that the composition $M_i\rightarrow M_{i+1}\rightarrow
    M$ is equivalent to the original map $M_i\rightarrow M$. Note that
    $\cofib(M_i\rightarrow M_{i+1})\we(F_i\otimes_{\pi_0R}R)[1]$, which is in
    $\Cscr$ (since $F_i$ was $(-1)$-coconnective). It follows that $M_{i+1}$ is
    in $\Cscr$. Additionally, $\pi_{-1}(F_i\otimes_{\pi_0R}R)\iso\pi_{-1}F_i\rightarrow\pi_{-1}M_i$
    is injective, so that $\pi_0M_i\iso\pi_0M_{i+1}$. By induction, we can extend this to a functor as claimed and
    the colimit is equivalent to $M$.\footnote{To construct a
    functor such as $\bN\rightarrow\Cscr_{/M}$ requires in theory an infinite amount
    of coherences. However, any sequence of composable maps in an
    $\infty$-category $\Dscr$ extends to a
    functor $\bN\rightarrow\Dscr$. This is because the inclusion
    $\mathrm{Spine}\,\bN\hookrightarrow\bN$ of the
    simplicial subset with a single $1$-simplex for each $i<i+1$
    and no nondegenerate higher simplices is inner anodyne. Indeed, it is a
    filtered colimit of inner anodyne inclusions
    $\mathrm{Spine}[n]\hookrightarrow[n]$;
    see~\cite[\href{https://kerodon.net/tag/00JA}{Tag 00JA}]{kerodon}.}

    To prove part (ii), note that a map $N\rightarrow N'$ is an equivalence in
    $\D(R)_{\leq 0}$ if and only if its fiber vanishes. As $M\otimes_R(-)$ is
    exact, we see that it is enough to check that if $N\in\D(R)_{\leq 0}$ is
    nonzero, then $M\otimes_RN$ is nonzero. Assume that $\pi_0N\neq 0$. Using the sequence $\{M_i\}$
    above, we obtain a sequence $M_0\otimes_R N\rightarrow M_1\otimes_R
    N\rightarrow\cdots$ whose colimit is $M\otimes_RN$. As $\cofib(M_i\otimes_R N\rightarrow
    M_{i+1}\otimes_RN)$ is in $\D(R)_{\leq 0}$, we see that $\pi_0(M_i\otimes_R
    N)\rightarrow\pi_0(M_{i+1}\otimes_RN)$ is injective. Now,
    $\pi_0(M_0\otimes_RN)\we\pi_0(M\otimes_{\pi_0R}R\otimes_RN)\iso\pi_0(M\otimes_{\pi_0R}N)\iso\pi_0M\otimes_{\pi_0R}\pi_0N$.
    This is nonzero by faithfulness of $\pi_0M$.
\end{proof}

\begin{corollary}\label{cor:absolutely flat}
    Let $R\rightarrow S$ be a morphism of coconnective $\bE_\infty$-rings. If
    $\pi_0R$ is absolutely flat, then $S$ is ccf over $R$. If additionally
    $\pi_0S$ is faithful as a $\pi_0R$-module, then $S$ is ccff over $R$.
\end{corollary}

\section{Boolean rings}\label{sec:boolean}

Throughout this section, fix a prime number $p$.
Say that a commutative $\bF_p$-algebra $A$ is $p$-Boolean if $x^p=x$ for all $x\in A$.
Below, $\varphi$ will always denote the $p$-power Frobenius $\varphi(x)=x^p$.
The category $\CAlg_{\bF_p}^{\varphi=1}$ of $p$-Boolean $\bF_p$-algebras
is equivalent to the opposite of the category of profinite spaces by Stone duality. In particular,
$\CAlg_{\bF_p}^{\varphi=1}$ is independent of $p$. See~\cite{stringall,antieau-spherical} for details.

Recall from the introduction that an anima $X$ is $p$-finite if $\pi_0X$ is finite, $X$ is truncated, and each homotopy group
$\pi_i(X,x)$ for $i\geq 1$ is a finite $p$-group. Let $\Sscr_{p\fin}\subseteq\Sscr$ be the full
subcategory of $p$-finite anima and let $\Pro(\Sscr_{p\fin})$ be its pro-category. 

There is an $\infty$-category $\DAlg_{\bF_p}^{\varphi=1}$ of $p$-Boolean derived commutative
$\bF_p$-algebras introduced in~\cite{antieau-spherical}. These are derived commutative $\bF_p$-algebras
equipped with a coherent trivialization of $\varphi$. Specifically, there is an action of $\B\bN$
on $\DAlg_{\bF_p}$ classifying the $p$-power Frobenius endomorphism and
$\DAlg_{\bF_p}^{\varphi=1}=\left(\DAlg_{\bF_q}\right)^{\h\B\bN}$, the $\infty$-category of
$\B\bN$-fixed points.

The main facts we need about $\DAlg_{\bF_p}^{\varphi=1}$ are (1) that the forgetful functor
$\DAlg_{\bF_p}^{\varphi=1}\rightarrow\DAlg_{\bF_p}$ preserves limits and colimits, (2) that the
continuous cochain functor $\Oscr\colon\Pro(\Sscr_{p\fin})^\op\rightarrow\DAlg_{\bF_p}^{\varphi=1}$ given by
$X\mapsto\Oscr(X)=\bF_p^X=\R\Gamma_\cont(X,\bF_p)$
is an equivalence, and (3) that the free $p$-Boolean derived commutative $\bF_p$-algebra on a
degree $-n$ generator, for $n\geq 1$, is equivalent to $\Oscr(K(\bF_p,n))=\bF_p^{K(\bF_p,n)}$, cochains on the Eilenberg--Mac Lane
anima $K(\bF_p,n)$.

\begin{theorem}\label{thm:boolean}
    Let $k$ be a $p$-Boolean $\bF_p$-algebra and let $R$ be a $p$-Boolean
    derived $k$-algebra such that the unit map $k\rightarrow\pi_0R$ is an isomorphism.
    Then, there is a retraction of $k\rightarrow R$ if
    \begin{enumerate}
        \item[{\em (i)}] there is a natural number $n$ such that
            $\pi_{-i}R=0$ for $1\leq i\leq n$ and either
            $k$ has cardinality at most $\aleph_n$ or $k\rightarrow R$ is $\aleph_n$-finitely
            presented;
        \item[{\em (ii)}] $\Spec k$ is extremally disconnected.
    \end{enumerate}
    Moreover, every such retraction $R\rightarrow k$ is ccff.
\end{theorem}

\begin{proof}
    The final claim follows from Proposition~\ref{prop:ccff}, so it is enough to produce the
    retractions in the given cases. In these cases, we claim
    that we can inductively kill negative degree classes of $R$ without changing $\pi_0$.

    By Proposition~\ref{prop:ccff}, $\bF_p^{K(\bF_p,m)}\rightarrow\bF_p$ is ccff for
    $m\geq 1$. It follows that $k^{K(\bF_p,m)}\rightarrow k$ is ccff as
    well, that $k^{K(\bF_p,m)}$ is the free $p$-Boolean derived
    commutative $k$-algebra on a degree $-m$ generator, and that
    $\pi_{-m}(k^{K(\bF_p,m)})\iso k$. We write this $p$-Boolean derived
    commutative $k$-algebra as $F_{m}(k)$. As $k$ is absolutely
    flat, every $k$-module $M$ is a filtered colimit of free $k$-modules; since
    ccff maps are closed under filtered colimits, $F_{m}(M)\rightarrow k$ is
    ccff for every discrete $k$-module $M$ and for all $m\geq 1$, where $F_{m}(M)$ denotes the free
    $p$-Boolean derived $k$-algebra on $M[-m]$. Moreover,
    $\pi_{-m}F_{m}(M)\iso M$ and $\pi_{-i}F_m(M)=0$ for $1\leq i < m$. 

    Every $\aleph_n$-compact morphism
    is a base change of a morphism with base having cardinality at most
    $\aleph_n$, so it suffices to prove (i) in the case that $k$ has cardinality
    $\aleph_n$. This condition guarantees that $k$ has global dimension at most $n+1$ by~\cite{osofsky}.
    Let $m$ be the least positive integer such that
    $\pi_{-m}R\neq 0$. By assumption, $m\geq n+1$ in case (i).
    Let $M=\pi_{-m}R$.
    We claim that under the hypotheses of the theorem, there is a lift of the
    identity on $\pi_{-m}R[-m]$ through a $k$-module map $\tau_{\geq
    -m}R\rightarrow\pi_{-m}R[-m]$. Indeed, the obstruction is a class of
    $\Ext^{m+1}(M,k)$ since the intermediate homotopy groups of $R$ vanish.
    If $k$ has global dimension at most $n+1$, then this group vanishes. If
    $\Spec k$ is extremally disconnected, then in fact $k$ is self-injective (see
    Lemma~\ref{lem:injective} below), so this group vanishes in that case as well.

    Composing some lift $\pi_{-m}R[-m]\rightarrow\tau_{\geq -m}R$ with
    $\tau_{\geq -m}R\rightarrow R$, we can form the pushout
    $$\xymatrix{
        F_m(M)\ar[d]\ar[r]& k\ar[d]\\
        R\ar[r]&R'
    }$$
    of $p$-Boolean derived commutative $k$-algebras,
    where top horizontal map sends $M$ to zero and the left vertical map is
    the canonical map sending induced by $\pi_{-m}M[-m]\rightarrow R$. All four maps in the diagram are ccff.
    By construction, the fiber of the left vertical map $F_m(M)\rightarrow R$
    is $(-m-1)$-coconnective as it induces isomorphisms on $\pi_0$ and
    $\pi_{-m}$. By base change along the ccff map $F_m(M)\rightarrow k$ we see
    that the fiber of $k\rightarrow R'$ is $(-m-1)$-coconnective. In
    particular, $k\iso\pi_0R'$ and $\pi_{-i}R'=0$ for $1\leq i\leq m-1$. On the other hand, $R\rightarrow R'$
    annihilates $\pi_{-m}$. Making an appropriate filtered colimit, we can kill
    off all negative homotopy groups of $R$ without changing $\pi_0$ by first
    annihilating $\pi_{-m}$ countably many times, and then $\pi_{-m-1}$, and so
    on.
\end{proof}

\begin{remark}
    Our original proof of Theorem~\ref{thm:boolean} deduced case (iii) from case (i).
    Indeed, suppose that $\Spec k$ is disconnected, let $X=\Spec k$, and for each $x\in X$, let
    $R_x=R\otimes_k\bF_p$. Since $\bF_p$ is flat over $k$, $\bF_p\iso\pi_0(R_x)$. Thus, we can find
    a map $f_x\colon R_x\rightarrow\bF_p$ by part (i). It follows that we can construct a
    commutative diagram
    $$\xymatrix{
        k\ar[r]\ar[d]&R\ar[d]&\\
        \prod_X\bF_p\ar[r]&\prod_X R_x\ar[r]^{\prod f_x}&\prod_X\bF_p,
    }$$
    where the bottom composition is the identity. As there is a retraction of
    $k\rightarrow\prod_X\bF_p$ by extremal disconnectedness, we can compose with this retraction to
    obtain $R\rightarrow\prod_X R_x\rightarrow\prod_X\bF_p\rightarrow k$, which is a retraction of
    $k\rightarrow R$.
\end{remark}

\begin{corollary}\label{cor:maps_to_completes}
    Suppose that $R$ is a $p$-Boolean derived ring. Then, there is a ccff map $R\rightarrow k$
    where $k$ is a discrete $p$-Boolean ring and $\Spec k$ is extremally disconnected.
\end{corollary}

\begin{proof}
    To see this, we can consider a map $\pi_0R\rightarrow k$
    where $k$ is a completion of the $p$-Boolean ring $\pi_0R$. The extension of scalars
    $R'=S\otimes_{\pi_0R}R$ is then ccff over $R$. Now, $k\iso\pi_0R'$ so we can apply
    Theorem~\ref{thm:boolean} to obtain a retraction $R'\rightarrow k$, which is necessarily ccff.
    The composition $R\rightarrow R'\rightarrow k$ is the desired ccff cover.
\end{proof}

\begin{theorem}[Generalized Gleason theorem]\label{thm:enough_projectives}
    Suppose that $k$ is a $p$-Boolean ring and that $\Spec k$ is extremally disconnected. If
    $k\rightarrow R$ is a map of $p$-Boolean derived rings where $k\rightarrow\pi_0R$ is injective,
    then there is a retraction $R\rightarrow k$.
\end{theorem}

\begin{proof}
    Using Corollary~\ref{cor:maps_to_completes}, there is a ccff map $R\rightarrow k'$ where $k'$ is
    discrete $p$-Boolean. The composition $k\rightarrow k'$ is ccff and hence injective, so we can
    conclude using the classical Gleason theorem~\cite{gleason} and Stone duality.
\end{proof}

 	For our next application, we need the following general notion.
 	
 	\begin{definition}
 	Let $\mathcal{C}$ be an $\infty$-category with finite limits. We say that $\mathcal{C}$ is regular if it admits geometric realizations of \v{C}ech nerves, and these are stable under base change. We say that a regular $\infty$-category $\mathcal{C}$ is coherent if for every object $X$ the poset of subobjects of $X$ has finite joins, and these are preserved by base change.
 	\end{definition}
 	
 Regular $\infty$-categories admit a well-behaved notion of effective epimorphism: these are the
 maps $f\colon X \rightarrow Y$ whose \v{C}ech nerve has geometric realization equivalent to $Y$. Similarly, in a coherent $\infty$-category there is a well-behaved notion of what it means for a finite family of maps $f_i\colon X_i \rightarrow Y$ to be jointly effectively epimorphic: we require that the join of the geometric realizations of the \v{C}ech nerves of $f_i$ is equal to $Y$.
  	
 	\begin{proposition}\label{prop:coherent}
 	The $\infty$-category $\Pro(\Sscr_{p\fin})$ is coherent. Furthermore, a morphism $f\colon X \rightarrow Y$ is an effective epimorphism if and only if the induced morphism of $p$-boolean derived commutative $\bF_p$-algebras $\mathcal{O}(Y) \rightarrow \mathcal{O}(X)$ is ccff. 
 	\end{proposition}
 	\begin{proof}
 	Since $\Sscr_{p\fin}$ is closed under finite limits and passage to subobjects inside $\Sscr$, we have that  $\Sscr_{p\fin}$ is regular.
        Applying \cite[Prop.~5.2.2]{stefanich-exact} we deduce that $\Pro(\Sscr_{p\fin})$ is regular.

	Since $\DAlg_{\bF_p}^{\varphi=1}$ has finite products which are stable by cobase change, we
        have that $\Pro(\Sscr_{p\fin})$ admits finite coproducts, which are stable under base change\footnote{This can also can be deduced formally from the fact that $\Sscr_{p\fin}$ has coproducts which are stable under base change.}.  	 	
   Suppose now given an object $X$ in $\Pro(\Sscr_{p\fin})$ and a pair of subobjects $U, V$ in $X$. Then the join of $U$ and $V$ exists, and is given by the image of the map $U \coprod V \rightarrow Y$. The fact that this is stable under base change follows from the fact that coproducts and image factorizations are both preserved by base change.
   
   It remains to establish the characterization of effective epimorphisms. Let  $f \colon X \rightarrow Y$ be an effective epimorphism in $\Pro(\Sscr_{p\fin})$. Let $I$ be the image of the induced map $\pi_0(\mathcal{O}(Y)) \rightarrow \pi_0(\mathcal{O}(X))$ and set
   \[
   A = \mathcal{O}(Y) \otimes_{\pi_0(\mathcal{O}(Y))} I
   \]
    so that the induced map $\mathcal{O}(Y) \rightarrow \mathcal{O}(X)$ factors as
   \[
   \mathcal{O}(Y) \rightarrow A \rightarrow \mathcal{O}(X).
   \]
   Here the first map is flat and surjective on $\pi_0$, and the second map induces a monomorphism
        on $\pi_0$, and therefore it is ccff by Corollary~\ref{cor:absolutely flat}. Now $\Spec A
        \rightarrow Y$ is a monomorphism through which $f$ factors, so we have that $\mathcal{O}(Y)
        = A$ and hence the map $\mathcal{O}(Y) \rightarrow \mathcal{O}(X)$ is ccff, as desired.
   
   Conversely, assume that the map $\mathcal{O}(Y) \rightarrow \mathcal{O}(X)$ is ccff.  We wish to show that $Y$ is the geometric realization of the \v{C}ech nerve of $f$. This amounts to showing that $\mathcal{O}(Y)$ is the totalization of the \v{C}ech nerve of $\mathcal{O}(Y) \rightarrow \mathcal{O}(X)$, which follows from the fact that ccff maps have descent for coconnective modules.
   	\end{proof}

\begin{remark}
    Let $X$ be a pro-$p$-finite anima realized as a cofiltered diagram $\{X_i\}_{i\in I}$ where
    each $X_i\in\Sscr_{p\fin}$. Then, $\{\pi_0 X_i\}_{i\in I}$ is a profinite set, which we will
    write as $\pi_0X$, and
    $\pi_0\Oscr(X)\iso\Cont(\pi_0X,\bF_p)=\bF_p^{\pi_0X}$. It follows from
    Propositions~\ref{prop:ccff} and~\ref{prop:coherent} that $X\rightarrow Y$ is an effective
    epimorphism if and only if $\pi_0X\rightarrow\pi_0Y$ is, which occurs if and only if
    $\pi_0X\rightarrow\pi_0Y$ is a surjection of profinite sets.
\end{remark}

   We equip $\Pro(\Sscr_{p\fin})$ with the coherent topology; in other words, this is the Grothendieck topology where a sieve is a covering sieve if and only if it contains a finite jointly effectively epimorphic family. We denote by $\Shv^\hyp(\Pro(\Sscr_{p\fin}))$ the corresponding $\infty$-category of hypercomplete sheaves of anima. This is not an $\infty$-topos, but we can make the same
definition as in the definition of condensed sets by taking a colimit over versions where we
restrict attention to $\kappa$-cocompact objects for larger and larger cardinals $\kappa$.
Let $\Cond\Sscr$ denote the $\infty$-category of condensed anima in the sense
of~\cite{scholze-condensed}. We have $\Cond\Sscr\we\Shv^\hyp(\Pro(\Fin))$, the $\infty$-category of
hypersheaves of anima on the category of profinite sets with the coherent topology.

\begin{corollary}\label{cor:condensed}
    The natural restriction functor $\Shv^{\hyp}(\Pro(\Sscr_{p\fin}))\rightarrow\Cond\Sscr$ is an
    equivalence.
\end{corollary}

\begin{proof}
    In light of Corollary~\ref{cor:maps_to_completes} and Theorem~\ref{thm:enough_projectives},
    it is enough to note that for uncountable strong limit cardinals $\kappa$ the covers constructed via
    Theorem~\ref{thm:boolean} of $\kappa$-cocompact pro-$p$-finite anima by extremally
    disconnected compact Hausdorff spaces are themselves $\kappa$-cocompact. Then, the result is a
    standard basis argument; see for example the proof of~\cite[Prop.~4.31]{bms2} (but use
    hypercovers) or~\cite[Cor.~A.7]{aoki-tensor}.
\end{proof}

\begin{remark}
    Corollary~\ref{cor:condensed} is an analogue of the Mathew--Mondal result~\cite{mathew-mondal} that identifies higher stacks with ccff
    stacks.
\end{remark}

\begin{remark}
    There is a `light' version of Corollary~\ref{cor:condensed}, where one considers only
    $\aleph_1$-small pro-systems, thanks to Theorem~\ref{thm:boolean}, which
    implies that any light pro-$p$-finite anima retracts onto its (light) profinite set of
    connected components.
\end{remark}

\begin{example}
    To evaluate a condensed anima $X$ on the pro-$p$-finite anima $T$, one
    realizes $T$ as $|S_\bullet|$ for a simplicial diagram of extremally disconnected compact
    Hausdorff spaces. Then, $X(T)\we\Tot(X(S_\bullet))$.
    Suppose now that $T$ is a compact
    Hausdorff space and similarly, write $T\we|S_\bullet|$ where $S_\bullet$ is a hypercover of $T$
    by extremally disconnected compact Hausdorff spaces. Then, the sheaf cohomology of $T$ with
    coefficients in $\bF_p$, $\C^*_{\mathrm{sheaf}}(T,\bF_p)$, is computed as
    $\Tot(\C^*_{\mathrm{sheaf}}(S_\bullet,\bF_p))\we\Tot(\bF_p^{S_\bullet})$.
    It is not the case that the natural map
    $\underline{T}\rightarrow\Spec(\C^*_{\mathrm{sheaf}}(T,\bF_p))$ is an equivalence, although it
    induces one on $\bF_p$-cohomology by~\cite[Thm.~3.2]{scholze-condensed}.
    For example, if $T=S^1$ with its usual topology, then $\underline{T}(*)\iso\Cont(*,S^1)\iso S^1$, as
    a set, while
    $\Spec(\C^*_{\mathrm{sheaf}}(T,\bF_p))(*)\we\Map_{\DAlg_{\bF_p}^{\varphi=1}}(\C^*_{\mathrm{sheaf}}(T,\bF_p),\bF_p)\we(\bT)_p^\wedge$, 
    the pro-$p$-finite completion of the anima $\bT$ underlying $S^1$.
\end{example}

\begin{corollary}\label{cor:boolean_descendable}
    Let $k$ be a Boolean ring and let $R$ be a $p$-Boolean derived $k$-algebra
    such that $\pi_0R$ is faithfully flat as a $k$-module. Then $k\rightarrow
    R$ is $\DAlg$-descendable (or even $p$-Boolean $\DAlg$-descendable) if
    \begin{enumerate}
        \item[{\em (i)}] $k$ has cardinality at most $\aleph_n$ or $k\rightarrow R$ is
            $\aleph_n$-compact for some natural number $n$, or
        \item[{\em (ii)}] $\Spec k$ is extremally disconnected.
    \end{enumerate}
\end{corollary}

\begin{proof}
    By~\cite{mathew-mondal}, the map $k\rightarrow\Tot(R^\bullet)$ is an
    equivalence.
    The fiber $I$ of $k\rightarrow R$ is $(-1)$-coconnective, so the fiber $I^{\otimes m+1}$ of
    $k\rightarrow\Tot^m(R^\bullet)$ is $(-m-1)$-coconnective as $k$ is absolutely flat.
    In particular, we can arrange under the hypotheses of the corollary that for
    $m$ sufficiently large, $\Tot^m(R^\bullet)$ satisfies the hypotheses of
    Theorem~\ref{thm:boolean}, so there is a $p$-Boolean derived $k$-algebra
    retraction of $k\rightarrow\Tot^m(R^\bullet)$, as desired.
\end{proof}

\begin{theorem}\label{thm:points}
    If $R$ is a nonzero derived commutative $\bF_p$-algebra such that $\pi_0R$
    is $p$-Boolean, then $R$ admits a derived commutative $\pi_0R$-algebra map
    to $\bF_p$.
\end{theorem}

\begin{proof}
    Possibly by replacing $R$ by $R_\perf$, we can assume that $R$ is nonzero perfect, coconnective, and
    $\pi_0R$ is $p$-Boolean. Note that
    $\pi_0(R\otimes_{\bF_p}R)\iso\pi_0R\otimes_{\bF_p}\pi_0R$. This ring is
    absolutely perfect by $p$-Booleanness. As $\pi_0R$ is $p$-Boolean, and hence derived
    $p$-Boolean, we can form the relative $p$-Booleanization of $R$:
    $$S=R\otimes_{R\otimes_{\pi_0R}R}R,$$
    where one map $R\otimes_{\pi_0R}R\rightarrow R$ is given by two copies of
    the identity and the other is given by the identity and Frobenius (which is
    naturally a $\pi_0R$-linear map since the Frobenius is the identity on
    $R$). The natural map $\pi_0(R\otimes_{\pi_0R}R)\rightarrow\pi_0R$ is an
    isomorphism. Thus, $R\otimes_{\pi_0R}R\rightarrow R$ is ccff by Proposition~\ref{prop:ccff} and the
    base change $R\rightarrow S$ is ccff as well. In particular, $S$ is a
    nonzero $p$-Boolean derived $\pi_0R$-algebra with $\pi_0R\iso\pi_0S$.
    From now on, let $k=\pi_0R\iso\pi_0S$. Pick a map $k\rightarrow \bF_p$
    and form the pushout $S'=S\otimes_k\bF_p$. By flatness, $\pi_0S'\iso\bF_p$
    and in particular $S'$ is non-zero. There is a retraction of
    $\bF_p\rightarrow S'$ by Theorem~\ref{thm:boolean}(i), so we are done.
\end{proof}

Say that an affine higher stack $\Xscr$ over $\bF_p$ is cohomologically connected if $\bF_p\iso\pi_0\R\Gamma(\Xscr,\Oscr)$.
As mentioned in the introduction, Kubrak pointed out that the following corollary is a consequence
of To\"en's~\cite[Thm.~2.4.5]{toen-affines}, which holds more generally over perfect fields.
Our proof is, however, different.

\begin{corollary}
    All cohomologically connected affine higher stacks over $\bF_p$ admit a rational point.
\end{corollary}

\begin{question}
    Is there a coconnective derived commutative $\bF_p$-algebra $R$ such that $\pi_0R$ has a map to
    $\bF_p$ but where the $p$-Booleanization of $R$ vanishes?
\end{question}

To conclude, we note the connection, used in the proof of Theorem~\ref{thm:boolean}, between self-injectivity of $R$ and extremal disconnectedness
of $\Spec R$ for $p$-Boolean $R$. These notions are also related to the condition that $R$ is
complete. Recall that a Boolean ring is complete if every nonempty subset has a supremum with
respect to the partial order where for $x,y\in R$ one has $x\leq y$ if and only if $xy=x$.
For $p$-Boolean rings, one imposes a partial order on the set of principal ideals in $R$ where
$(x)\leq (y)$ if and only $(xy)=(x)$. Ko Aoki explained to us the proof of (iii) implies (i).

\begin{lemma}\label{lem:injective}
    Let $R$ be a $p$-Boolean ring. The following conditions are equivalent:
    \begin{enumerate}
        \item[{\em (i)}] $R$ is complete;
        \item[{\em (ii)}] $\Spec R$ is extremally disconnected;
        \item[{\em (iii)}] $R$ is self-injective.
    \end{enumerate}
\end{lemma}

\begin{proof}
    The equivalence between (i) and (ii) is given in~\cite[22.4]{sikorski}.
    The implication (i) implies (iii) is~\cite[Cor.~4]{brainerd-lambek}.
    The converse (iii) implies (i) is~\cite[Sec.~2.4]{lambek}.
    For the reader's convenience, we will show here that (i) implies (ii) implies (iii) implies (i)
    in the case where $p=2$, the classical Boolean case. The general case is similar.
    To do this, we will use both characterizations of extremally disconnected compact Hausdorff spaces: that
    they are the projective objects in compact Hausdorff spaces and that they are the compact
    Hausdorff spaces for which the closure of every open set is open;
    see~\cite[\href{https://stacks.math.columbia.edu/tag/08YN}{Tag 08YN}]{stacks-project}.
    Assume (i). Let $X=\Spec R$ and let $U\subseteq X$ be open. Write $U$ as a union
    $U=\bigcup_{a\in\Lambda}X_a$ of basic open subsets of $X$. Here, $X_a$ denotes the set of
    primes not containing $a$. We will also write $V(b)$ for the closed subset of primes 
    containing $b$. Every closed subset of $X$ is an intersection of closed subsets of the form
    $V(b)$, so the closure $\overline{U}$ of $U$ is described as $$\overline{U}=\bigcap_{b\in\Gamma}V(b),$$
    where $\Gamma$ is the set of $b$ such that $X_a\subseteq V(b)$ for all $a\in\Lambda$.
    As $X_a=V(1-a)$, we have $X_a=V(1-a)\subseteq V(b)$ if and only if $b\leq
    1-a$ if and only if $a\leq 1-b$. Let $c$ be the supremum of the set $\Lambda$,
    which exists by the assumption that $R$ is complete. Then, $a\leq c\leq 1-b$ for all
    $a\in\Lambda$ and $b\in\Gamma$. Thus, $X_a\subseteq V(1-c)\subseteq V(b)$
    for all $a\in\Lambda$ and $b\in\Gamma$. It follows that $\overline{U}=V(1-c)=X_c$, which
    is clopen. Hence, (ii).

    Assume (ii). Consider the natural map $R\rightarrow\prod_X\bF_p$, where the product is over all
    points of the prime spectrum and the composite $R\rightarrow\bF_p$ to a factor corresponding to
    a prime is precisely the reduction modulo that prime. This map is injective as its kernel is
    the nilradical of $R$, which is zero. Thus, $\Spec \prod_X\bF_p\rightarrow X$ is
    surjective. By extremal disconnectedness, this map admits a section which implies in particular
    that $R$ is a retract of $\prod_X\bF_p$ as an $R$-module.
    Now, since each $R\rightarrow\bF_p$ is flat and $\bF_p$
    is injective as a module over itself, each $\bF_p$ is injective as an $R$-module. Thus,
    $\prod_X\bF_p$ is injective as an $R$-module, so $R$ is injective as it is a retract of an
    injective module. This proves (iii).

    Claim: a subset $Z$ of $R$ has a supremum $a$ in $R$ if and only if $\ann(Z)$ is principal and
    equal to $(1-a)$. If $a=\sup(Z)$, then $z\leq a$ for all $z\in Z$, so that $az=z$ and hence
    $(1-a)z=0$ for all $z\in Z$. Thus, $(1-a)\in\ann(Z)$. If $b\in\ann(Z)$, then $bz=0$ for all
    $z\in Z$, so $(1-b)z=z$ for all $z\in Z$. Thus, $z\leq 1-b$ for all $z\in Z$. Hence, $a\leq
    1-b$, or $1-a\geq b$. Thus, $b\in(1-a)$ and we have shown that $\ann(Z)$ is principal. The
    converse is similar. Thus, every annihilator is principal if and only if $R$ is complete.
    Now, in order to prove (iii) implies (i), it is enough to show under the assumption that $R$ is self-injective that every
    annihilator is principal.

    Let $Z\subseteq R$ be a subset and let $I$ be the ideal generated by $Z$. Note that
    $\ann(I)=\ann(Z)$. The intersection of $I$ and $\ann(I)$ is $(0)$ so that the natural inclusion
    map $I\oplus\ann(I)\rightarrow R$ is injective. Consider $f\colon I\oplus\ann(I)\rightarrow R$,
    defined by $f(x,z)=x$. By self-injectivity of $R$ there exists $a\in R$ such that
    $f(x,z)=x=a(x+z)$ for all $x\in I$ and $z\in\ann(I)$. In particular, $ax=x$ for $x\in I$
    and $az=0$ for $z\in\ann(I)$. It follows that $I\subseteq(a)$ and thus
    $(1-a)=\ann(a)\subseteq\ann(I)$, which implies $\ann(\ann(I))\subseteq\ann(1-a)=(a)$. But, we
    also have that $a\in\ann(\ann(I))$ by construction,
    which proves that $\ann(\ann(I))=(a)$. Now, $\ann(I)=\ann(\ann(\ann(I)))=\ann(a)=(1-a)$ is
    principal. Here we used the $3$-Boolean property of annihilators.
\end{proof}

\begin{example}
    A countable, non-finite $p$-Boolean ring is not self-injective. Indeed, it is not complete for cardinality
    reasons.
\end{example}

\section{Faithfully flat maps in characteristic $0$}\label{sec:noetherian}

We show in this section that if $R$ is a regular Noetherian $\bQ$-algebra of finite Krull dimension, then every faithfully
flat map $R\rightarrow S$ is $\bE_\infty$-descendable.

\begin{definition}
    Let $R$ be a commutative ring. A map $R\rightarrow S$ of $\bE_\infty$-rings is $n$-good if its fiber has
    Tor-amplitude bounded above by $-n$.
\end{definition}

\begin{example}
    A flat map $R\rightarrow S$ of commutative rings is $0$-good. A faithfully flat map
    $R\rightarrow S$ of commutative rings is $1$-good.
\end{example}

\begin{example}
    An extension of scalars of an $n$-good map is $n$-good.
\end{example}

We recall the following lemma from~\cite[Lem.~2.13]{mathew-mondal}.

\begin{lemma}\label{lem:tramplitude2}
    Suppose that $R$ is a commutative ring and that $S$ is a derived commutative $R$-algebra with
    an augmentation $S\rightarrow R$. Assume moreover that the fiber of the augmentation has
    Tor-amplitude bounded above by $-1$ as an $R$-module. If an
    $S$-module $M$ has Tor-amplitude bounded above by $n$ as an $R$-module, then it has
    Tor-amplitude bounded above by $n$ as an $S$-module.
\end{lemma}

\begin{lemma}\label{lem:ngood_induct}
    Suppose that $R$ is a commutative $\bQ$-algebra and that $R\rightarrow S$ is $n$-good where
    $n\geq 1$.
    If $\LSym_R Q[-m]\rightarrow S$ is any map where $m\geq n$ and $Q$ is a flat $R$-module, then
    the pushout $T$, defined by the commutative diagram
    $$\xymatrix{
        \LSym_R Q[-m]\ar[r]\ar[d]&S\ar[d]\\
        R\ar[r]&T,
    }$$
    is $n$-good over $R$.
\end{lemma}

\begin{proof}
    Consider the commutative diagram
    $$\xymatrix{
        R\ar@{=}[r]\ar[d]&R\ar[d]\\
        \LSym_R Q[-m]\ar[r]&S
    }$$
    and take horizontal and vertical fibers to obtain a commutative diagram
    $$\xymatrix{
        F_1\ar[r]\ar[d]&F_2\ar[r]\ar[d]&F_3\ar[d]\\
        0\ar[r]\ar[d]&R\ar@[=][r]\ar[d]&R\ar[d]\\
        F_1[1]\ar[r]&\LSym_R Q[-m]\ar[r]&S
    }$$
    whose rows and columns are fiber sequences.
    By hypothesis, $F_3$ has Tor-amplitude bounded above by $-n$ as an $R$-module and one checks that $F_2$ has
    Tor-amplitude bounded above by $-m-1$ as an $R$-module. Since $m\geq n$, it
    follows that $F_1$ has Tor-amplitude bounded above by $-n-1$ and hence that $F_1[1]$ has
    Tor-amplitude bounded above by $-n$ as an $R$-module. By Lemma~\ref{lem:tramplitude2}, it
    follows that $F_1[1]$ has Tor-amplitude bounded above by $-n$ as an $\LSym_RQ[-m]$-module and
    hence
    $F_1[1]\otimes_{\LSym_RQ[-m]}R$ has Tor-amplitude bounded above by $-n$ as an $R$-module by
    base change. This is equivalent to the fiber of $R\rightarrow T$, so $T$ is $n$-good.
\end{proof}

\begin{lemma}\label{lem:in_the_limit}
    Let $\bQ\subseteq R$ be a commutative ring and suppose that $R\rightarrow S$ is $n$-good for some $n\geq 1$. Then, there is an extension
    $S\rightarrow W$ such that
    \begin{enumerate}
        \item[{\em (a)}] $W$ is $n$-good over $R$,
        \item[{\em (b)}] $\pi_iW=0$ for $i\neq 0,-n+1$, and
        \item[{\em (c)}] $\pi_{-n+1}W$ is a flat $R$-module.
    \end{enumerate}
    If $n=1$, then $W$ is discrete and $R\rightarrow W$ is faithfully flat.
    If $n\geq 2$, then $R\iso\pi_0W$.
\end{lemma}

\begin{proof}
    Inductively kill classes in degrees $-n$ and lower using
    Lemma~\ref{lem:ngood_induct} to obtain a sequential diagram $S\rightarrow S'\rightarrow
    S''\rightarrow\cdots$
    where every class in degree at most $-n$ vanishes at some point. The colimit $W$ remains $n$-good
    as an $R$-algebra by exactness of filtered colimits. On the other hand, it is also
    $(-n+1)$-connective. Thus, the fiber of $I$ of $R\rightarrow W$ is $(-n)$-connective and has
    Tor-amplitude bounded above by $(-n)$. In other words, $I[n]$ is a flat $R$-module.
    If $n\geq 2$, we are done. If $n=1$, then there is an exact sequence $0\rightarrow R\rightarrow
    W\rightarrow I[1]\rightarrow 0$, which shows that $W$ is faithfully flat.
\end{proof}

\begin{definition}
    Say that an $n$-good map $R\rightarrow S$ is $n$-excellent if $S$ satisfies the conclusions 
    of Lemma~\ref{lem:in_the_limit}. The proof of the lemma shows that if $S$ is $n$-good over $R$
    and $(-n+1)$-connective, then it is $n$-excellent.
\end{definition}

\begin{lemma}\label{lem:last_class}
    Suppose that $\bQ\subseteq R$ is a commutative ring and that $S$ is an $n$-excellent
    $\bE_\infty$-$R$-algebra
    for some $n\geq 2$. If the extension class in $\Ext^n_R(\pi_{-n+1}S,R)$ defined by $S$
    vanishes, then any pushout $T$ defined by
    $$\xymatrix{
        \LSym_R\pi_{-n+1}S[-n+1]\ar[r]\ar[d]&S\ar[d]\\
        R\ar[r]&T,
    }$$
    where the implicit map $\pi_{-n+1}S\rightarrow\pi_{-n+1}S$ is the identity,
    is $n$-good over $R$.
\end{lemma}

\begin{proof}
    Let $Q=\pi_{-n+1}S$.
    The fiber of $\LSym_R Q[-n+1]\rightarrow S$ consists of $\oplus_{i\geq
    2}\LSym^i(Q[-n+1])$, which has Tor-amplitude bounded over $R$ above by $-2n+2$.
    Thus, the same is true for the fiber of $R\rightarrow T$ using Lemma~\ref{lem:tramplitude2}
    and base change. As $n\geq 2$, $-2n+2\leq -n$, so we are done.
\end{proof}

\begin{theorem}\label{thm:noetherian}
    If $R$ is a Noetherian $\bQ$-algebra of finite Krull dimension, then every faithfully
    flat map $R\rightarrow S$ is $\bE_\infty$-descendable.
\end{theorem}

\begin{proof}
    Let $d=\max(1,\dim R)\geq 1$.
    By~\cite[Cor.~3.2.7]{raynaud-gruson}, the projective dimension of any flat module
    $R$-module is at most $d$.
    Let $R\rightarrow S^\bullet$ be the \v{C}ech nerve of $R\rightarrow S$. Let
    $I=\fib(R\rightarrow S)$. Since $R\rightarrow S$ is faithfully flat, $I[1]$ is a flat
    $R$-module.
    Then, $R\rightarrow\Tot^{d} S^\bullet$ has fiber given by $I^{\otimes d+1}$. It is
    hence $(d+1)$-good. As $d+1\geq 2$,
    by alternating the constructions of Lemma~\ref{lem:in_the_limit} and Lemma~\ref{lem:last_class},
    and using that the extension class condition of the latter lemma always holds in this case by
    the projective dimension bound for flats, we
    obtain a transfinite composition of $\bE_\infty$-algebras over $\Tot^d S^\bullet$ whose colimit is $R$.
\end{proof}

The proof of Theorem~\ref{thm:noetherian} used the technique of flattening by blowing up of
Raynaud--Gruson. We can prove a weaker bound on the projective dimension of flat modules under
stronger hypotheses. We include a proof for the interested reader. It can be used to prove the
corollary under the hypothesis that $R$ is of finite type over a characteristic $0$ field.

\begin{lemma}\label{lem:finite_dimension}
Let $k$ be a field, and let $R$ be a finite type $k$-algebra of Krull dimension at most $d$. If $M$
    is a flat $R$-module, then the projective dimension of $M$ is at most $d+1$.
\end{lemma}
\begin{proof}
We argue by induction on $d$. The case $d < 0$ (i.e, $R = 0$) is clear, so assume from now on that $d \geq 0$ and that the lemma is known for $d-1$.

Let $N$ be a discrete $R$-module. We need to show that $\Hom_R(M, N)$ is $(-d-1)$-connective. The module $N$ admits a finite filtration whose successive quotients $N_j$ are modules over the reduction $R_{\text{red}}$ of $R$. It suffices to show that $\Hom_R(M, N_j)$ is $(-d-1)$-connective for all $j$. Replacing $N$ by $N_j$ we may assume that $N$ is an $R_{\text{red}}$-module. We then have 
\[
\Hom_R(M, N) = \Hom_{R_{\text{red}}}(M \otimes_R R_{\text{red}}, N).
\]
Replacing $R$ by $R_{\text{red}}$ we may now reduce to the case when $R$ is reduced.

Let $f$ be a nonzerodivisor of $R$ with the property that $R[f^{-1}]$ is regular over $k$, and denote by $N^\wedge = \lim (N \otimes_R R/f^mR)$ the (derived) $f$-adic completion of $N$. We have an exact sequence
\[
\Hom_R(M, F) \rightarrow \Hom_R(M, N) \rightarrow \Hom_R(M, N^\wedge)
\]
where $F$ is the fiber of the map $N \rightarrow N^\wedge$.  Note that $F$ is a $(-1)$-connective $R[f^{-1}]$-module.  The first term in the above sequence is equal to  $\Hom_{R[f^{-1}]}(M \otimes_R R[f^{-1}], F )$, so it is $(-d-1)$-connective virtue of the fact that $R[f^{-1}]$ is a regular Noetherian ring of dimension at most $d$. To prove the lemma it now suffices to show that 
\[
\Hom_R(M, N^\wedge) = \lim \Hom_R(M, N \otimes_R R/f^mR)
\]
is $(-d-1)$-connective. To see this it is enough to show that $\Hom_R(M, N \otimes_R R/f^mR)$ is $(-d)$-connective for all $m \geq 0$. This $R$-module admits a finite filtration whose successive quotients are all equal to 
\[
\Hom_R(M, N \otimes_R R/fR) = \Hom_{R/fR}(M \otimes_R R/fR, N \otimes_R R/fR).
\]
We finish by observing that the above is $(-d)$-connective by virtue of our inductive hypothesis.
\end{proof}

\section{Some finite morphisms}\label{sec:finite}

Bhatt and Scholze prove in~\cite{bhatt-scholze-affine} that $h$-covers $X\rightarrow Y$ of
Noetherian schemes induce
descendable morphisms $\Oscr_Y\rightarrow f_*\Oscr_X$ in $\D_\qc(Y)$. We do not know if they
induce $\bE_\infty$-descendable morphisms. We show below that some interesting examples of
$h$-covers do induce $\bE_\infty$-descendable morphisms.

\begin{proposition}\label{prop:h}
    Let $g\colon R\rightarrow S$ be an integral extension of Noetherian commutative rings. If
    $\dim R\leq 1$, then $g$ is $\DAlg$-descendable. If $\dim R\leq 2$ and $g$ is
    birational, then $g$ is $\DAlg$-descendable.
\end{proposition}

\begin{proof}
    We can use primary decomposition to reduce to the case where $R$ is a domain and $S$ is a
    finite product of domains. Suppose that
    $(0)=\qfrak_1\cap\cdots\cap\qfrak_n$ is a primary decomposition of $R$ where each $\qfrak_i$ is
    $\pfrak_i$-primary for a prime ideal $\pfrak_i$. Let $\afrak=\qfrak_2\cap\cdots\cap\qfrak_n$.
    Then, since $\qfrak_1\cap\afrak=(0)$, $R$ is a pullback
    $$\xymatrix{
        R\ar[r]\ar[d]&R/\qfrak_1\ar[d]\\
        R/\afrak\ar[r]&R/(\qfrak_1,\afrak).
    }$$
    In particular, $R\rightarrow R/\qfrak_1\times R/\afrak$ is $\DAlg$-descendable by
    Corollary~\ref{cor:pullback}. Inductively, it follows that $R\rightarrow\prod_{i=1}^n
    R/\qfrak_i$ is $\DAlg$-descendable. Moreover, since the kernels of $R/\qfrak_i\rightarrow
    R/\pfrak_i$ are nilpotent, it follows that $R\rightarrow\prod_{i=1}^n R/\pfrak_i$ is
    $\DAlg$-descendable. Using Lemma~\ref{lem:two_three}, it suffices now to check that each
    $R/\pfrak_i\rightarrow S\otimes_R R/\pfrak_i$ is descendable (where we use the non-derived
    tensor product).

    Using the same kind of decomposition on $S\otimes_R R/\pfrak_k$, we can assume
    that $R$ is a domain and that $S$ is a product of finite $R$-algebras, each of which is a
    domain. By faithfulness and the going up theorem, we can further reduce to the case that $R\rightarrow S$ is a
    map of domains and that $S$ is a finite $R$-algebra.
    If $\dim R=0$, then $R\rightarrow S$ is faithfully flat, so the result follows by
    Theorem~\ref{thm:noetherian}. If $\dim R=1$, let $R'$ and $S'$ be the normalizations of
    $R$ and $S$ in their fields of fractions. Then, $R'\rightarrow S'$ is faithfully flat since
    $R'$ is a Dedekind domain. Using the commutative square
    $$\xymatrix{
        R\ar[r]\ar[d]&S\ar[d]\\
        R'\ar[r]&S'
    }$$ and Lemma~\ref{lem:two_three}, it is thus enough to show that $R\rightarrow R'$ is
    $\DAlg$-descendable. There is some function $f\in R$ such that $R[1/f]\iso R'[1/f]$.
    We assume that $f$ is chosen so that in fact the fiber $I$ of $R\rightarrow R'$ admits the
    structure of an $R/f$-module. This is possible because $R\rightarrow R'$ is injective so that $I\we R'/R[-1]$.

    From this point, the remainder of the proof in the $\dim R=1$ case is the same as the proof in
    the $\dim R=2$ case. That is, we assume that there is some $f\in R$ such that $R[1/f]\we
    S[1/f]$, that $I=\fib(R\rightarrow S)$ admits the structure of an $R/f$-module, and that
    $R/f\rightarrow S/f$ is $\DAlg$-descendable. (For surfaces, the hypothesis that
    $R/f\rightarrow S/f$ is $\DAlg$-descendable follows from the $1$-dimensional case we are establishing simultaneously.)

    We claim that the canonical map $R\rightarrow R/f\times_{S//f}S$ admits a retraction.
    To see this, note that
    the map $$\pi_0(R/f\times_{S//f}S)\rightarrow S$$ factors through the pullback
    $R/f\times_{S/f}S\rightarrow S$. An element of $R/f\times_{S/f}S$ consists of a pair
    $(\overline{r},s)$ where $\overline{r}\in R/f$, $s\in S$, and $g(\overline{r})\equiv s\pmod f$
    in $S/f$. If $r$ is a lift of $\overline{r}$ to $R$, we can rewrite the condition as saying
    that $g(r)\equiv s\pmod f$, so that there is some $t\in S$ such that $r=s+ft$. Our assumption
    that $S/R$ is an $R/f$-module implies that $ft\in R\subseteq S$. Thus, $s=r-ft\in R$.
    Therefore, the projection map $$R/f\times_{S//f}S\rightarrow S$$ factors through $R\subseteq S$.
    Now, since $R/f\rightarrow S/f$ is $\DAlg$-descendable, so is $R/f\rightarrow S//f$. It
    follows that the smallest full subcategory $\Cscr\subseteq\CAlg(\D(R))$ containing the objects which
    admit a map from $S$ and closed under finite limits and retracts contains $S$, $S//f$, $R/f$ (by
    $\DAlg$-descendability of $R/f\rightarrow S//f$), the pullback $R/f\times_{S//f}S$ by
    closure under finite limits, and finally $R$ by closure under retracts.
\end{proof}

\section{Tannaka duality}\label{sec:tannaka}

Throughout this section we work with prestacks (i.e, presheaves of anima) on the $\infty$-category of connective $\bE_\infty$-rings. A quasicompact geometric stack with affine diagonal is a prestack which satisfies fpqc descent, has affine diagonal, and admits a faithfully flat surjection from an affine scheme.

\begin{theorem}\label{thm:tannaka}
Let $X$ be a quasicompact geometric stack with affine diagonal, and assume that there exists an
    fpqc surjection $p\colon U \rightarrow X$ with $U$ affine such that the commutative algebra
    $p_*(\mathcal{O}_U)$ in $\QCoh(X)$ is $\bE_\infty$-descendable. Then, for every affine scheme
    $S$, the morphism of anima
\[
\Hom(S, X) \rightarrow \Hom_{\Pr^L_{\CAlg(\D(\bS))/}}( \CAlg(\QCoh(X)) , \CAlg(\QCoh(S)))
\]
given by sending $f$ to $f^*$ is an embedding whose image consists of those functors which preserve finite limits and connective objects.
\end{theorem}

\begin{example}\label{ex:scheme}
Let $X$ be a quasicompact semiseparated scheme. Then $X$ satisfies the hypothesis of Theorem~\ref{thm:tannaka}. Indeed, if $\lbrace U_\alpha \rbrace$ is a finite affine open cover of $X$ then setting $U = \coprod U_\alpha$ we find that $p_* \mathcal{O}_U$ is $\bE_\infty$-descendable, by a variant of Corollary~\ref{cor:zariski}.
\end{example}

\begin{example}\label{ex:BG}
Let $X= BG$ be the classifying stack of an affine algebraic group over a field $k$ of characteristic zero.
    Then $X$ satisfies the hypothesis of Theorem~\ref{thm:tannaka}, by taking $p$ to be the
    projection $U = \Spec k \rightarrow X$. The fact that $p_* \mathcal{O}_U$ is $\bE_\infty$-descendable follows from an adaptation of the arguments from Section~\ref{sec:noetherian}. More generally, Theorem~\ref{thm:tannaka} applies whenever $X$ is a quasicompact geometric stack with affine diagonal over $\bQ$ satisfying the following two conditions:
\begin{enumerate}
\item[{(i)}] every truncated connective quasicoherent sheaf on $X$ receives a surjection from a flat sheaf;
\item[{(ii)}] there exists an integer $d$ such that $\Ext^n(\mathcal{F}, \mathcal{O}_X) = 0$ for every $n \geq d$ and every flat sheaf $\mathcal{F}$ on $X$.
\end{enumerate}
We note that condition (i) is implied by the resolution property, while condition (ii) holds in the case when $X$ is an Artin stack and of finite type over a field, see~\cite[Prop.~8.3.2]{gaitsgory-shvcat}.
\end{example}

\begin{example}
    Let $q\colon X' \rightarrow X$ be a quasicompact semiseparated schematic morphism between geometric
    stacks, and assume that $X$ satisfies the hypothesis of Theorem~\ref{thm:tannaka}. Then the
    theorem also applies to $X'$. To see this, let $p\colon U \rightarrow X$ be an fpqc surjection from
    an affine scheme such that $p_*(\mathcal{O}_U)$ is $\bE_\infty$-descendable, and let $p'\colon U'
    \rightarrow X'$ be its base change to $X'$, so that $U'$ is a quasicompact semiseparated scheme. Let
    $p''\colon V \rightarrow U'$ be an fpqc surjection from an affine scheme such that
    $p''_*(\mathcal{O}_V)$ is $\bE_\infty$-descendable (which exists by Example~\ref{ex:scheme}).
    Then $(p' \circ p'')_* (\mathcal{O}_V)$ is $\bE_\infty$-descendable by a version of
    Lemma~\ref{lem:two_three}.

    In particular, combining with Example~\ref{ex:BG} we deduce that quotients of quasicompact
    semiseparated schemes by actions of affine algebraic groups in characteristic zero satisfy the
    conditions of Theorem~\ref{thm:tannaka}.
\end{example}

\begin{remark}
Our proof of Theorem~\ref{thm:tannaka} depends on the usual Tannaka duality results concerning symmetric monoidal functors between $\infty$-categories of quasicoherent sheaves. Extending this proof to the context of derived algebraic geometry (that is, geometry built on spectra of connective derived commutative rings) would require proving a version of Tannaka duality in the derived context, which is beyond the aims of this paper. 

Nevertheless, we can still prove a weakened version of Theorem~\ref{thm:tannaka} in the derived
    setting: the assignment $f \mapsto f^*$ is an embedding, and its image consists of those
    functors which preserve finite limits, connective objects, and flat objects. This follows by a
    variant of the argument that proves Theorem~\ref{thm:tannaka}. More precisely, the assignment
    $f \mapsto f^*$ is first encoded as a map of prestacks $\mu_X : X \rightarrow
    \Spec^\cn(\DAlg(X))$, which is shown to be an embedding using a variant of
    Lemmas~\ref{lem:affine} and~\ref{lem:pullback} below. One then argues, by a variant of the
    arguments in Lemma~\ref{lem:nuX}, that $\mu_X$ is an fpqc surjection onto the subobject of
    $\Spec^\cn(\DAlg(X))$ determined by the preservation of flatness condition.
\end{remark}

\begin{remark}
    As we shall see below in Lemma~\ref{lem:affine}, in the case when $X$ is an affine scheme the
    conclusion of Theorem~\ref{thm:tannaka} holds without requiring finite limit preservation. This
    condition is however necessary for more general classes of stacks.

    For instance, let $G$ be a finite group and let $X=BG$ be the classifying stack of $G$ over a
    characteristic zero field $k$ (so that in particular the conditions of
    Theorem~\ref{thm:tannaka} are satisfied, by Example~\ref{ex:BG}). Let $q\colon BG \rightarrow
    \Spec k$ be the projection. Then the pullback functor $q^*\colon \CAlg(\Spec k) \rightarrow
    \CAlg(\QCoh(BG))$ admits a left adjoint $q_!\colon \CAlg(\QCoh(BG)) \rightarrow
    \CAlg(\QCoh(\Spec k))$. Note that $q^*$ is fully faithful (since already the pullback functor
    $\QCoh(\Spec k) \rightarrow \QCoh(BG)$ is) and consequently $q_!$ provides a functor in
    \[
    \Hom_{\Pr^L_{\CAlg(\D(\bS))/}}( \CAlg(\QCoh(BG)) , \CAlg(\QCoh(\Spec k))).
    \]
    Note that $q_!$ is equivalent to the colimit functor 
    \[
    \CAlg(\QCoh(\Spec k))^{BG} \rightarrow \CAlg(\QCoh(\Spec k))
    \]
    and in particular it sends connective algebras to connective algebras. However $q$ does not
    arise from a map $\Spec k \rightarrow BG$. To see this, consider the projection $p\colon \Spec k
    \rightarrow BG$. Then $q_!(p_*k)$ is the algebra of functions on the fixed points of the action
    of $G$ on $\Spec(p^*p_*(k)) = G$, and vanishes since $G$ acts freely on itself. However the
    pullback of $p_*k$ along any map $\Spec k\rightarrow BG$ does not vanish.
\end{remark}

The rest of this section will be devoted to the proof of Theorem~\ref{thm:tannaka}. In what follows
we will also make use of nonconnective prestacks, i.e., presheaves of anima on the $\infty$-category of $\bE_\infty$-rings.

\begin{notation}
Let $X$ be a prestack. We denote by $\Spec(\QCoh(X))$ the nonconnective prestack which sends each nonconnective affine scheme $S$ to 
\[
\Hom_{\CAlg(\Pr^L)}( \QCoh(X),\QCoh(S)).
\]
Similarly, we denote by $\Spec(\CAlg(\QCoh(X)))$ the nonconnective prestack that sends each nonconnective affine scheme $S$ to the subanima of
\[
\Hom_{\Pr^L_{\CAlg(\D(\bS))/}}( \CAlg(\QCoh(X)) , \CAlg(\QCoh(S)))
\]
on those functors which which preserve finite limits. We denote by
\[
\Spec(\QCoh(X)) \xrightarrow{\nu_X} \Spec(\CAlg(\QCoh(X)))
\]
the canonical map obtained by taking commutative algebra objects.
\end{notation}

\begin{lemma}\label{lem:affine}
Let $A$ and $B$ be $\bE_\infty$-rings. Then the morphism of anima
\[
\Hom(A, B) \rightarrow \Hom_{\Pr^L_{\CAlg(\D(\bS))/}}( \CAlg(\D(A)) , \CAlg(\D(B)))
\]
given by sending an $A$-algebra $C$ to $B \otimes_A C$, is an equivalence.
\end{lemma}
\begin{proof}
Passing to right adjoints, we may identify the right hand side with
\[
\Hom_{\Pr^R_{/\CAlg(\D(\bS))}}( \CAlg(\D(B)) , \CAlg(\D(A))).
\]
The desired assertion follows from the fact that the projections 
\[
\CAlg(\D(A)) \rightarrow \CAlg(\D(\bS)) \leftarrow \CAlg(\D(B))
\]
are left fibrations corepresented by $A$ and $B$.
\end{proof}

\begin{lemma}\label{lem:affine2}
If $X$ is an affine scheme, then $\nu_X$ is an isomorphism.
\end{lemma}
\begin{proof}
This follows from Lemma~\ref{lem:affine}.
\end{proof}

\begin{lemma}\label{lem:pullback}
Let $X$ be a prestack, and let $p\colon U \rightarrow X$ and $q\colon V \rightarrow X$ be affine morphisms. Then we have equivalences
\[
\Spec(\QCoh(U \times_X V)) = \Spec(\QCoh(U)) \times_{\Spec(\QCoh(X))} \Spec(\QCoh(V))
\]
 and
\[
\Spec(\CAlg(\QCoh(U \times_X V))) = \Spec(\CAlg(\QCoh(U))) \times_{\Spec(\CAlg(\QCoh(X)))} \Spec(\CAlg(\QCoh(V)))
\]
\end{lemma}
\begin{proof}
For the first equation it suffices to show that we have an equivalence
\[
\QCoh(U \times_X V) = \QCoh(U) \otimes_{\QCoh(X)} \QCoh(V).
\]
Indeed, both sides can be identified with the $\infty$-category of modules over $p_*(\mathcal{O}_U) \otimes q_*(\mathcal{O}_V)$ in $\QCoh(X)$.

Similarly, to establish the second equation it suffices to show that we have an equivalence
\[
\CAlg(\QCoh(U \times_X V)) = \CAlg(\QCoh(U)) \otimes_{\CAlg(\QCoh(X))} \CAlg(\QCoh(V)),
\]
where we equip all $\infty$-categories in sight with their cocartesian symmetric monoidal
    structures and take the relative tensor product in the symmetric monoidal $\infty$-category of
    $\infty$-categories having sifted colimits. Indeed, both sides can be identified with the
    $\infty$-category of modules over $p_*(\mathcal{O}_U) \otimes q_*(\mathcal{O}_V)$ in
    $\CAlg(\QCoh(X))$.
\end{proof}

\begin{lemma}\label{lem:nuX}
Let $X$ be a quasicompact geometric stack with affine diagonal, and assume that there exists an
    fpqc surjection $p\colon U \rightarrow X$ with $U$ affine such that the commutative algebra $p_*(\mathcal{O}_U)$ in $\QCoh(X)$ is $\bE_\infty$-descendable. Then $\nu_X$ is an equivalence.
\end{lemma}
\begin{proof}
We have a commutative square
\[
\begin{tikzcd}
\Spec(\QCoh(U)) \arrow{r}{\nu_U} \arrow{d}{q} & \Spec(\CAlg(\QCoh(U))) \arrow{d}{q'} \\
\Spec(\QCoh(X)) \arrow{r}{\nu_X} & \Spec(\CAlg(\QCoh(X)))
\end{tikzcd}
\]
Applying lemmas  \ref{lem:affine2} and \ref{lem:pullback}  we see that the \v{C}ech nerves of $q$
    and $q'$ are equivalent. Since both $\Spec(\QCoh(X))$ and $\Spec(\CAlg(\QCoh(X)))$ satisfy
    descent with respect to the $\bE_\infty$-descendable topology (and in fact the descendable
    topology), we may reduce to showing that $q$ and $q'$
    are $\bE_\infty$-descendable surjections. We will argue that this is the case for $q'$; the case of $q$ is analogous.

Fix a map $S \rightarrow \Spec(\CAlg(\QCoh(X)))$ with $S$ a nonconnective affine scheme,
    corresponding to a left adjoint functor $F\colon \CAlg(\QCoh(X)) \rightarrow \CAlg(\QCoh(S))$ under
    $\CAlg(\D(\bS))$ which preserves finite limits. We wish to show that the base change of $q'$ to
    $S$ is a descendable surjection. Unwinding the definitions, we see that the base change of $q'$
    to $S$ is given by the spectrum of $F(p_*(\mathcal{O}_U))$. The desired assertion now follows
    from the fact that $F$ preserves $\bE_\infty$-descendable algebras, since it preserves finite
    limits and coproducts.
 \end{proof}

\begin{proof}[Proof of Theorem~\ref{thm:tannaka}]
Let $\Spec^\cn(\QCoh(X))$ and $\Spec^\cn(\CAlg(\QCoh(X)))$ be the subobjects of the restriction of
    $\Spec(\QCoh(X))$ and  $\Spec(\CAlg(\QCoh(X)))$ to connective $\bE_\infty$-rings consisting of those functors that preserve connective objects. Then we have morphisms
\[
X \xrightarrow{\eta^\cn_X}  \Spec^\cn(\QCoh(X)) \xrightarrow{\nu^\cn_X} \Spec^\cn(\CAlg(\QCoh(X))).
\]
Our goal is to show that the composition $\nu^\cn_X \circ \eta^\cn_X$ is an equivalence. We claim
    that in fact both $\eta^\cn_X$ and $\nu^\cn_X$ are equivalences. The fact that $\eta^\cn_X$ is
    an equivalence follows from \cite[Thm.~1.0.3]{stefanich-tannaka}. By Lemma~\ref{lem:nuX}, to
    see that $\nu^\cn_X$ is an equivalence it is enough to show that, for every affine scheme $S$,
    a colimit preserving symmetric monoidal functor $F\colon \QCoh(X) \rightarrow \QCoh(S)$
    preserves connective objects if the induced functor $\CAlg(\QCoh(X)) \rightarrow
    \CAlg(\QCoh(S))$ preserves connective objects. Indeed, if $M$ is a connective object of
    $\QCoh(X)$, then the square zero algebra $F(\mathcal{O}_X \oplus M) \we \mathcal{O}_S \oplus F(M)$ is connective, from which we deduce that $F(M)$ is connective.
\end{proof}

\small
\bibliographystyle{amsplain}
\bibliography{eid}

\medskip
\medskip
\noindent
\textsc{Department of Mathematics, Northwestern University}\\
{\ttfamily antieau@northwestern.edu}

\medskip
\noindent
\textsc{Max-Planck-Institut für Mathematik Bonn}\\
{\ttfamily stefanich@mpim-bonn.mpg.de}

\end{document}